\documentclass[final,1p,times,10pt]{elsarticle}
\usepackage{lmodern}
\usepackage{setspace}
\usepackage[labelsep=period]{caption}

\DeclareCaptionFormat{myformat}{\fontsize{7.23}{7.23}\selectfont#1#2#3}
\usepackage{amssymb}

\newtheorem{theorem}{Theorem}
\newtheorem{corollary}{Corollary}
\newtheorem{proposition}{Proposition}
\newtheorem{lemma}{Lemma}
\newtheorem{definition}{Definition}
\newproof{proof}{Proof}
\newcommand{\contract}{\mathord{\varparallelinv}}

\journal{Advances in Mathematics}

\begin{document}

\begin{frontmatter}

\title{The diameter of associahedra}

\author[liafa,efrei]{Lionel Pournin}
\ead{lionel.pournin@liafa.univ-paris-diderot.fr}
\address[liafa]{LIAFA, Universit{\'e} Paris Diderot, Case 7014, 75205 Paris Cedex 13, France}
\address[efrei]{EFREI, 30-32 avenue de la R{\'e}publique, 94800 Villejuif, France}

\begin{abstract}
It is proven here that the diameter of the $d$-dimensional associahedron is $2d-4$ when $d$ is greater than $9$. Two maximally distant vertices of this polytope are explicitly described as triangulations of a convex polygon, and their distance is obtained using combinatorial arguments. This settles two problems posed about twenty-five years ago by Daniel Sleator, Robert Tarjan, and William Thurston.
\end{abstract}

\end{frontmatter}

\section{Introduction}

\label{Asection.1}

The history of associahedra can be traced back to 1951, when Dov Tamari introduced the lattice of bracketed expressions \cite{Tam51} that was to be named after him. Tamari observed that the graph of this lattice is likely to be that of a polytope and he sketched the associahedra of dimensions $1$, $2$, and $3$, as reported in \cite{Sta12}. A decade later, the same structure was rediscovered within the study of homotopy associative $H$-spaces and constructed as a cellular ball by James Stasheff \cite{Sta63}. However, it is only in 1984 that Mark Haiman constructed associahedra as convex polytopes in a manuscript that remained unpublished \cite{Hai84}. The first published construction of associahedra is due to Carl Lee \cite{Lee89} and many other constructions and generalizations have been found since then \cite{Bil90,Ceb13,Cha02,Gel90,Hol07,Lod04,Pos09}. Today, these polytopes are among the primary examples of interest in the introductory chapters of several textbooks on discrete geometry \cite{DRS10,Zie95}. Associahedra and their face complexes attract much attention not only due to their importance in discrete geometry and algebraic topology, but also because of their connection with a number of combinatorial objects such as triangulations of convex polygons \cite{Lee89,Sle88} or binary trees \cite{Deh10,Sle88}. The diameter of associahedra is especially relevant to operations that can be carried out to transform any of these objects into other objects of the same type: flips for triangulations and rotations for binary trees. These operations provide alternatives to the associativity rule originally used to define the Tamari lattice \cite{Hug78,Tam51}.

About twenty-five years ago, while working on the dynamic optimality conjecture, Daniel Sleator, Robert Tarjan, and William Thurston proved the following result \cite{Sle88}:
\begin{theorem}\label{Atheorem.1}
The $d$-dimensional associahedron has diameter at most $2d-4$ when $d$ is greater than $9$, and this bound is sharp when $d$ is large enough.
\end{theorem}

Unfortunately, Theorem \ref{Atheorem.1} does not provide any clue on the smallest dimension above which the diameter of an associahedron is always twice its dimension minus four. This diameter has been calculated up to dimension $15$ in \cite{Sle88} using a computer program. It was found that, within this range, only associahedra of dimensions $8$ and $10$ to $15$ would satisfy the aforementioned rule, leading to the following conjecture 
\cite{Sle88}: \emph{the diameter of the $d$-dimensional associahedron is $2d-4$ when $d$ is greater than $9$}. The problem remained open until now, and made its way to textbooks \cite{Dev11,DRS10}.

While the statement of Theorem \ref{Atheorem.1} is purely combinatorial, the proof given in \cite{Sle88} is based on a construction in hyperbolic space. At the end of their paper, the authors comment that the role played by geometry in the problem ``may seem mysterious'' and they add that there should exist a proof using only combinatorial arguments. These observations led to consider the existence of a combinatorial proof for Theorem \ref{Atheorem.1} as an open problem on its own \cite{Deh10,Dev11}. Recently, Patrick Dehornoy made progress towards such a proof by obtaining a lower bound of the form $2d-O(\sqrt{d})$ on the diameter of the $d$-dimensional associahedron using combinatorial arguments \cite{Deh10}.

The two problems mentioned in the preceding paragraphs are solved in this article. In other words, a combinatorial proof is given that the $d$-dimensional associahedron has diameter $2d-4$ when $d$ is greater than $9$.

This result can be rephrased using either triangulations or binary trees instead of associahedra. Indeed, the diameter of the $d$-dimensional associahedron is also the maximal flip distance between two triangulations of a convex polygon with $d+3$ vertices and the maximal rotation distance between two binary trees with $d+1$ internal nodes \cite{Deh10}. In fact, the proof given in this article makes extensive use of the notions of triangulations and flips. These notions are now defined. Consider a convex polygon $\pi$. A set of two distinct vertices of $\pi$ will be referred to as an \emph{edge} on $\pi$. The two elements of an edge will also be called its \emph{vertices}. Two edges on $\pi$ are \emph{crossing} if they are disjoint but their convex hulls are non-disjoint. An edge on $\pi$ is called an \emph{interior edge} on $\pi$ if it crosses some edge on $\pi$, and it is called a \emph{boundary edge} on $\pi$ otherwise.

\begin{definition}\label{Adefinition.2}
Let $\pi$ be a convex polygon. A triangulation of $\pi$ is a set of pairwise non-crossing edges on $\pi$ that is maximal with respect to the inclusion. 
\end{definition}

Consider a triangulation $T$ of a convex polygon $\pi$. It follows immediately from the above definition that all the boundary edges on $\pi$ belong to $T$. Every such edge will be called a boundary edge \emph{of} $T$ and the other elements of $T$ will be called the interior edges \emph{of} $T$. The next proposition is a classical characterization of triangulations using the number of their edges:

\begin{proposition}\label{Aproposition.2}
Let $\pi$ be a convex polygon with $n$ vertices. A set of pairwise non-crossing edges on $\pi$ is a triangulation of $\pi$ if and only if it has cardinality $2n-3$. 
\end{proposition}

Flips and flip-graphs are now defined. Let $T$ be a triangulation of a convex polygon $\pi$. Note that any interior edge $\varepsilon$ of $T$ is one of the diagonals of a unique convex quadrilateral whose four boundary edges belong to $T$. Call $\varepsilon'$ the other diagonal of this quadrilateral. A new triangulation of $\pi$, called $T\mathord{/}\varepsilon$, is built replacing $\varepsilon$ by $\varepsilon'$ within $T$:
$$
T\mathord{/}\varepsilon=[T\mathord{\setminus}\{\varepsilon\}]\cup\{\varepsilon'\}
$$

This operation is referred to as the \emph{flip} that removes $\varepsilon$ from $T$ and introduces $\varepsilon'$ into $T$, or simply as the operation of \emph{flipping edge $\varepsilon$ in triangulation $T$}. The graph whose vertices are the triangulations of $\pi$ and whose edges connect two triangulations whenever they can be obtained from one another by a flip is called the \emph{flip-graph of $\pi$}. This graph is also the $1$-skeleton of the associahedron. More precisely, consider a convex polygon with $n$ vertices and call $\Sigma_n$ the abstract simplicial complex whose faces are the sets of pairwise non-crossing interior edges on this polygon. The following definition is due to Mark Haiman \cite{Hai84} and Carl Lee \cite{Lee89}:
\begin{definition}\label{Adefinition.3}
Any $d$-dimensional polytope whose boundary complex is isomorphic to $\Sigma_{d+3}$ is referred to as the $d$-dimensional associahedron.
\end{definition}

It immediately follows that the $1$-skeleton of the $d$-dimensional associahedron is isomorphic to the flip-graph of any convex polygon with $d+3$ vertices. Hence, results on the flip-graph of convex polygons can be transposed to the $1$-skeleton of associahedra and vice versa. In particular, the flip-graph of a convex polygon $\pi$ is always connected: any triangulation of $\pi$ can be obtained by performing a sequence of flips from any other triangulation of $\pi$. Note that this connectedness property also follows from a much simpler argument in the case of convex polygons (see for instance \cite{Sle88}), and that flip-graphs are not always connected in more general cases \cite{DRS10}.

The main result in this article will be obtained showing that two triangulations $A_n^-$ and $A_n^+$ of a convex polygon with $n$ vertices have flip distance $2n-10$ when $n$ is greater than $12$. These triangulations are depicted in Fig. \ref{Afigure.3}. In order to find their flip distance, the following recursive inequality will be established for all $n>12$, where $A_n$ denotes the pair $\{A_n^-,A_n^+\}$ and $\delta(A_n)$ the flip distance between $A_n^-$ and $A_n^+$:
$$
\delta(A_n)\geq\min(\{\delta(A_{n-1})+2,\delta(A_{n-2})+4,\delta(A_{n-5})+10,\delta(A_{n-6})+12\})\mbox{.}
$$

The values of $\delta(A_n)$ will be calculated explicitly when $n$ ranges from $3$ to $12$ and the proof will then proceed by induction on $n$. Observe that the above recursive inequality compares flip distances for polygons of different sizes. Methods to do so will be given in Section \ref{Asection.2}, as well as generalizations of a lemma from \cite{Sle88} that allows to prescribe flips along a geodesic path between two triangulations.

Pair $A_n$ and two other pairs of triangulations that will serve for the statement of intermediate results are constructed in Section \ref{Asection.6}. Some properties regarding the flip distance of $A_n$ will also be given in this section. The above recursive inequality is proven in Section \ref{Asection.7} using the results of Section \ref{Asection.2}. The diameter of associahedra is obtained in Section \ref{Asection.8} based on this recursive lower bound. In section \ref{Asection.9}, a third related open problem, posed in \cite{Deh10} is solved and a discussion on the possibility of using the same ideas to find similar results for other flip-graphs completes the article.

\section{Results on flip distances}
\label{Asection.2}

Several results on the flip-graphs of convex polygons are proven in \cite{Sle88} using combinatorial arguments. Some of these results are stated in the first part of this section and one of them is further generalized. 

\begin{definition}
Let $U$ and $V$ be two triangulations of some convex polygon. A path of length $k$ from $U$ to $V$ is a sequence $(T_i)_{0\leq{i}\leq{k}}$ of triangulations so that $T_0=U$, $T_k=V$, and whenever $0\leq{i}<k$, triangulation $T_{i+1}$ is obtained from $T_i$ by a flip.
\end{definition}

A path from a triangulation $U$ to a triangulation $V$ can alternatively be thought of as a sequence of $k$ flips that transform $U$ into $V$. A path from $U$ to $V$ is called \emph{geodesic} if its length is minimal among all the paths from $U$ to $V$.

\begin{definition}
Let $U$ and $V$ be two triangulations of some convex polygon. The flip distance of $U$ and $V$ is the length of any geodesic path between these triangulations. This flip distance is hereafter denoted by $\delta(\{U,V\})$.
\end{definition}

Lemma 2 from \cite{Sle88} provides the following upper bound on flip distances:
\begin{lemma}\label{Alemma.1}
Consider a convex polygon $\pi$ with $n$ vertices. If $n$ is greater than $12$, then the flip distance of any two triangulations of $\pi$ is not greater than $2n-10$.
\end{lemma}

This lemma can be reformulated as the first part of Theorem \ref{Atheorem.1}: \emph{the diameter of the $d$-dimensional associahedron is not greater than $2d-4$ when $d$ is greater than $9$}. However, as discussed in \cite{Deh10,DRS10,Sle88}, finding sharp lower bounds on flip distances is the difficult task in the problem at hand. Lemma 3 from \cite{Sle88} will be helpful in the search for such lower bounds. The original statement of this lemma is divided into two parts denoted (a) and (b). Part (a) can be rephrased as follows:

\begin{lemma}\label{Alemma.2}
Let $U$ and $V$ be two triangulations of some convex polygon. If an edge of $V$ can be introduced in triangulation $U$ by a flip, then there exists a geodesic path from $U$ to $V$ that begins with this flip.
\end{lemma}

Part (b) states that, when an edge has been removed along a geodesic path, it cannot reappear later along the same path:

\begin{lemma}\label{Alemma.3}
Consider a geodesic path $(T_i)_{0\leq{i}\leq{k}}$ between two triangulations $U$ and $V$. If $\varepsilon$ is an edge of both $U$ and $V$, then $\varepsilon$ is an edge of $T_i$ whenever $0\leq{i}\leq{k}$.
\end{lemma}

Two generalizations of Lemma \ref{Alemma.2} are now given. The first one provides a weaker condition under which the first flip of a geodesic path can be prescribed:
\begin{lemma}\label{Alemma.4}
Let $(T_i)_{0\leq{i}\leq{k}}$ be a geodesic path between two triangulations $U$ and $V$. If for some $j\in\{0, ..., k\}$, an edge of $T_j$ can be introduced into $U$ by a flip, then there exists a geodesic path from $U$ to $V$ that begins with this flip.
\end{lemma}
\begin{proof}
Consider an integer $j\in\{0, ..., k\}$ and assume that some edge $\varepsilon$ of $T_j$ can be introduced in $U$ by a flip. It is further assumed without loss of generality that $U=T_0$ and $V=T_k$. Note that every geodesic path between $U$ and $T_j$ has length $j$. In this case, according to Lemma \ref{Alemma.2}, there exists a geodesic path $(T'_i)_{0\leq{i}\leq{j}}$ from $U$ to $T_j$ whose first flip introduces $\varepsilon$ into $U$. For every $i\in\{j+1, ..., k\}$, denote $T'_i=T_i$. It follows that $(T'_i)_{0\leq{i}\leq{k}}$ is a path of length $k$ from $U$ to $V$ that begins with the desired flip. Since every path of length $k$ from $U$ to $V$ is geodesic, the lemma is proven. \qed
\end{proof}

This lemma can in turn be generalized to sequences of flips. The next theorem states that, under similar requirements, there exists a path between two triangulations of a convex polygon with its first several flips prescribed.

\begin{theorem}\label{Atheorem.3}
Let $(T_i)_{0\leq{i}\leq{k}}$ be a geodesic path and $(T'_i)_{0\leq{i}\leq{l}}$ a path so that $T_0$ and $T'_0$ coincide. If, for some $j\in\{0, ..., k\}$, $T_j$ contains all the edges introduced along path $(T'_i)_{0\leq{i}\leq{l}}$, then triangulations $T'_l$ and $T_k$ have flip distance $k-l$.
\end{theorem}
\begin{proof}
The proof will proceed by induction on $l$. First observe that if $l$ is equal to $1$, then the statement of the theorem simplifies to that of Lemma \ref{Alemma.4}. As a consequence, the desired result holds in this case.

Now assume that the result holds for a given positive integer $l$. Consider a geodesic path $(T_i)_{0\leq{i}\leq{k}}$ and a path $(T'_i)_{0\leq{i}\leq{l+1}}$ so that $T_0$ is equal to $T'_0$. Assume that for some integer $j\in\{0, ..., k\}$, $T_j$ contains all the edges introduced along path $(T'_i)_{0\leq{i}\leq{l+1}}$. By induction, $T'_l$ and $T_j$ have flip distance $j-l$. Hence, according to Lemma \ref{Alemma.2}, the flip distance of $T'_{l+1}$ and $T_j$ is $j-l-1$. Now consider a geodesic path from $T'_{l+1}$ to $T_j$ and denote by $T'_{l+i+1}$ the triangulation found after $i$ flips along this path. Since this path has length $j-l-1$, triangulations $T'_j$ and $T_j$ coincide. Further denote $T'_i=T_i$ when $j+1\leq{i}\leq{k}$. According to this construction $(T'_i)_{0\leq{i}\leq{k}}$ is a path of length $k$ from $T_0$ to $T_k$. Since every path of length $k$ from $T_0$ to $T_k$ is geodesic then so is $(T'_i)_{0\leq{i}\leq{k}}$. As $T'_k$ coincides with $T_k$, triangulations $T'_{l+1}$ and $T_k$ have flip distance $k-l-1$. \qed
\end{proof}

\subsection{Deleting a vertex from the triangulations of a convex polygon}
\label{Asection.4}

Inequalities on flip distances will be stated at the end of this section, based on the operation of deleting a vertex from the triangulations of a convex polygon $\pi$. Informally, this operation consists in displacing a vertex of $\pi$ to its clockwise successor. Carrying out a deletion within a triangulation of $\pi$ results in a triangulation of a polygon $\pi'$ with one vertex less than $\pi$. It will be shown that this operation induces a continuous surjection from the flip-graph of $\pi$ onto the flip-graph of $\pi'$, and therefore makes it possible to compare the distances in these graphs.

Note that the deletion operation has been defined in greater generality for the triangulations of cyclic polytopes by J{\"o}rg Rambau (see \cite{Ram97} and Section 6.1. in \cite{DRS10}). 

Consider a boundary edge $\{a,b\}$ on a convex polygon $\pi$ and assume that $a$ immediately precedes $b$ in the clockwise ordering of $\pi$. In this case, pair $(a,b)$ will be called a \emph{clockwise oriented boundary edge on $\pi$}. For any subset $\varsigma$ of $\pi$, let $\varsigma{\contract}a$ denote the set obtained by substituting $b$ for $a$ within $\varsigma$. Note that, if $a$ does not belong to $\varsigma$, then $\varsigma{\contract}a=\varsigma$. Now consider a triangulation $T$ of $\pi$. The operation of deleting $a$ from $T$ consists in replacing $a$ by $b$ within every edge of $T\mathord{\setminus}\{\{a,b\}\}$. The resulting set is denoted by $T{\contract}a$:
$$
T{\contract}a=\{\varepsilon{\contract}a:\varepsilon\in{T\mathord{\setminus}\{\{a,b\}\}}\}\mbox{.}
$$

The notation $\contract$ is preferred here to the notation $\setminus$ used in \cite{DRS10,Ram97} because it allows to clearly distinguish deletions from the relative complement operator on sets.

Since $\{a,b\}$ is first removed from $T$ when $a$ is deleted from this triangulation, all the elements of $T{\contract}a$ are edges on $\pi{\contract}a$. It turns out that $T{\contract}a$ is a triangulation of $\pi{\contract}a$. This statement is proven in the more general case of cyclic polytopes in \cite{DRS10,Ram97}. A simpler proof is provided here for the special case of convex polygons:
\begin{proposition}\label{Aproposition.3}
Consider a convex polygon $\pi$ with at least $4$ vertices. Let $a$ be some vertex of $\pi$. If $T$ is a triangulation of $\pi$, then $T{\contract}a$ is a triangulation of $\pi{\contract}a$.
\end{proposition}
\begin{proof}
Observe that $\pi{\contract}a$ has one vertex less than $\pi$. As $\pi$ has at least $4$ vertices then $\pi{\contract}a$ is still a polygon. Now consider a triangulation $T$ of $\pi$. Denote by $b$ the vertex of $\pi$ that immediately follows $a$ clockwise. Since $a$ and $b$ are consecutive along the boundary of $\pi$, replacing $a$ by $b$ within two non-crossing edges on $\pi$ will never result in a pair of crossing edges. Hence, $T{\contract}a$ is a set of pairwise non-crossing edges on $\pi{\contract}a$.

Now observe that there is a unique vertex $c$ of $\pi$ so that $\{a,c\}$ and $\{b,c\}$ are edges of $T$. These two edges are the only distinct elements of $T\mathord{\setminus}\{\{a,b\}\}$ that are transformed into the same edge when $a$ is deleted from $T$. Hence $T{\contract}a$ has one element less than $T\mathord{\setminus}\{\{a,b\}\}$, and two elements less than $T$. As, in addition, $\pi{\contract}a$ has one vertex less than $\pi$, it immediately follows from Proposition \ref{Aproposition.2} that $T{\contract}a$ is a triangulation of $\pi{\contract}a$. \qed
\end{proof}

The way flip-graphs and paths within them react to deletions is now investigated. Let $\pi$ be a convex polygon and $(a,b)$ a clockwise oriented boundary edge on $\pi$. Further consider a triangulation $T$ of $\pi$. As already mentioned in the proof of Proposition \ref{Aproposition.3}, there is a unique vertex $c$ of $\pi$ so that $\{a,c\}$ and $\{b,c\}$ are edges of $T$. Vertex $c$ is denoted by $\mathrm{lk}_T(\{a,b\})$ hereafter, and referred to as the \emph{link of $\{a,b\}$ in $T$}.
\begin{definition}
Consider a clockwise oriented boundary edge $(a,b)$ on a convex polygon $\pi$. A flip that transforms a triangulation $U$ of $\pi$ into a triangulation $V$ is called incident to edge $\{a,b\}$ if this edge has distinct links in $U$ and in $V$.
\end{definition}

The following theorem is identical to the third assertion of Lemma 5.13 from \cite{Ram97}, proven in the more general case of cyclic polytopes. A simpler proof is given here in the particular case at hand. Let $\pi$ be a convex polygon with at least four vertices and $(a,b)$ a clockwise oriented boundary edge on $\pi$. This theorem states that, if $(T_i)_{0\leq{i}\leq{k}}$ is a path between two triangulations $U$ and $V$ of $\pi$, then there exists a path between $U{\contract}a$ and $V{\contract}a$ shorter than $(T_i)_{0\leq{i}\leq{k}}$ by the number of flips incident to $\{a,b\}$ along $(T_i)_{0\leq{i}\leq{k}}$. This path is found by deleting $a$ from $T_0$, ..., $T_k$ and subsequently removing unnecessary triangulations from the sequence thus obtained.

\begin{theorem}\label{Atheorem.4}
Let $\pi$ be a convex polygon with at least $4$ vertices and $(a,b)$ a clockwise oriented boundary edge on $\pi$. Further consider two triangulations $U$ and $V$ of $\pi$. If $j$ flips are incident to edge $\{a,b\}$ along a path of length $k$ between  $U$ and $V$, then there exists a path of length $k-j$ between $U{\contract}a$ and $V{\contract}a$.
\end{theorem}
\begin{proof}
Assume that $j$ flips are incident to $\{a,b\}$ along a path $(T_i)_{0\leq{i}\leq{k}}$ from $U$ to $V$. For every $i\in\{1, ..., k\}$, denote by $q_i$ the quadrilateral whose diagonals are exchanged by the $i$-th flip along this path. Denote by $\varepsilon_i$ the diagonal of $q_i$ that belongs to $T_{i-1}$ and by $\varsigma_i$ the one that belongs to $T_i$. Observe that the only edge of $T_{i-1}{\contract}a$ that may not belong to $T_i{\contract}a$ is $\varepsilon_i{\contract}a$. Similarly, the only edge of $T_i{\contract}a$ that may not belong to $T_{i-1}{\contract}a$ is $\varsigma_i{\contract}a$. Hence, $T_{i-1}{\contract}a$ and $T_i{\contract}a$ are either identical or they can be obtained from one another by a flip that exchanges $\varepsilon_i{\contract}a$ and $\varsigma_i{\contract}a$.

First assume that the $i$-th flip along the path $(T_i)_{0\leq{i}\leq{k}}$ is not incident to $\{a,b\}$. In this case, $q_i{\contract}a$ is still a quadrilateral whose diagonals are $\varepsilon_i{\contract}a$ and $\varsigma_i{\contract}a$.  By construction, these edges respectively belong to triangulation $T_{i-1}{\contract}a$ and to triangulation $T_i{\contract}a$. As $\varepsilon_i{\contract}a$ and $\varsigma_i{\contract}a$ are crossing, $T_{i-1}{\contract}a$ and $T_i{\contract}a$ are necessarily distinct, proving that they are obtained from one another by a flip.

Now assume that the $i$-th flip along path $(T_i)_{0\leq{i}\leq{k}}$ is incident to $\{a,b\}$. In this case, $\{a,b\}$ is a boundary edge on $q_i$ and $q_i{\contract}a$ is a triangle. Observe that every edge on this triangle can be obtained by substituting $b$ for $a$ within some boundary edge on $q_i$. Since every boundary edge on $q_i$ is contained in both $T_{i-1}$ and $T_i$, all the edges on $q_i{\contract}a$ necessarily belong to $T_{i-1}{\contract}a$ and to $T_i{\contract}a$. As $\varepsilon_i{\contract}a$ and $\varsigma_i{\contract}a$ are edges on $q_i{\contract}a$, they are contained in both $T_{i-1}{\contract}a$ and $T_i{\contract}a$. Hence, these triangulations are identical.

This shows that one can build a path from $U{\contract}a$ to $V{\contract}a$ that successively visits triangulations $T_i{\contract}a$, where $i$ ranges from $0$ to $k$, skipping the values of $i$ so that the $i$-th flip along $(T_i)_{0\leq{i}\leq{k}}$ is incident to $\{a,b\}$. This path has length $k-j$, where $j$ is the number of flips incident to $\{a,b\}$ along $(T_i)_{0\leq{i}\leq{k}}$, and the proof is complete. \qed
\end{proof}

Theorem \ref{Atheorem.4} is the most important result of the section because it relates the flip distance of two triangulations with that of two triangulations with fewer vertices.
\begin{figure}
\begin{centering}
\includegraphics{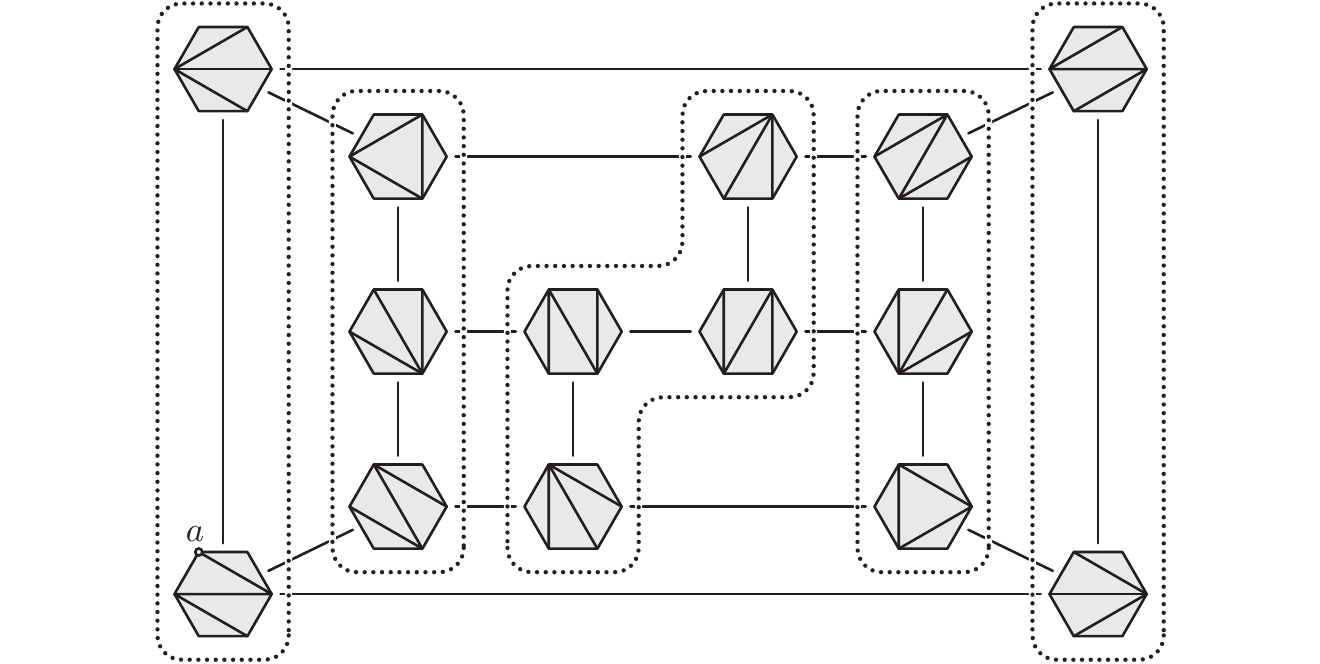}
\caption{The flip-graph of the hexagon depicted with solid lines and the equivalence classes of $\sim_a$ shown with dotted lines, where $a$ is the top left vertex of the hexagon.}\label{Afigure.1}
\end{centering}
\end{figure}
A more general result on flip-graphs can be obtained from the main arguments in the proof of this theorem: consider some vertex $a$ of a convex polygon $\pi$. If $U$ and $V$ are two triangulations of $\pi$, denote $U\mathord{\sim_a}{V}$ whenever $U{\contract}a$ and $V{\contract}a$ are identical. This defines an equivalence relation $\sim_a$ on the set $\tau$ of all the triangulations of $\pi$. As shown in the proof of Theorem \ref{Atheorem.4}, deleting $a$ from every triangulation of $\pi$ preserves the adjacency of any two triangulations that are not equivalent by $\sim_a$. It immediately follows that \emph{the flip-graph of $\pi{\contract}a$ is isomorphic to the graph obtained by contracting every element of $\tau/\mathord{\sim_a}$ in the flip-graph of $\pi$}. This property is illustrated in Fig. \ref{Afigure.1}. In this figure, the flip-graph of the hexagon is represented using solid lines. The dotted lines show the five elements of $\tau/\mathord{\sim_a}$, where $a$ is the top left vertex of the hexagon, as indicated in the figure. One can see that contracting the elements of $\tau/\mathord{\sim_a}$ in the flip-graph of the hexagon results in a cycle whose vertices are the five triangulations of the pentagon and whose edges correspond to the flips between them. This property will not be explicitly needed in the following, but it provides some intuition on the meaning of Theorem \ref{Atheorem.4}.

\subsection{Deletion-induced inequalities on flip distances}

Let $(a,b)$ be a clockwise oriented boundary edge on a convex polygon $\pi$. Consider two triangulations $U$ and $V$ of $\pi$ and denote by $P$ the pair $\{U,V\}$. The operation of deleting vertex $a$ from $P$ amounts to deleting vertex $a$ from both $U$ and $V$, resulting in the following pair of triangulations:
$$
P{\contract}a=\{U{\contract}a,V{\contract}a\}\mbox{.}
$$

Now consider the path between triangulations $U$ and $V$ used in the statement of Theorem \ref{Atheorem.4}. Observe that no assumption is made on the length of this path. Choosing a geodesic path makes it possible to compare the flip distance of pair $P$ with that of pair $P{\contract}a$. More precisely, Theorem \ref{Atheorem.4} then states that the difference $\delta(P)-\delta(P{\contract}a)$ is bounded below by the number of flips incident to $\{a,b\}$ along any geodesic path between $U$ and $V$. This suggests the following quantity should be singled out:
\begin{definition}
Let $U$ and $V$ be two triangulations of a convex polygon $\pi$ and $(a,b)$ a clockwise oriented boundary edge on $\pi$. The maximal number of flips incident to $\{a,b\}$ along any geodesic path between $U$ and $V$ will be denoted by $\vartheta(\{U,V\},a)$.
\end{definition}

Using this notation, the next result is an immediate consequence of Theorem \ref{Atheorem.4}:
\begin{corollary}\label{Acorollary.1}
Let $\pi$ be a convex polygon with at least $4$ vertices and $P$ a pair of triangulations of $\pi$. For any vertex $a$ of $\pi$, the following inequality holds:
$$
\delta(P)\geq\delta(P{\contract}a)+\vartheta(P,a)\mbox{.}
$$
\end{corollary}

If $\vartheta(P,a)$ is at least $2$, this corollary states that $\delta(P)$ is not less than $\delta(P{\contract}a)+2$, which is exactly the kind of inequality needed to prove the main result of this paper. One can find a vertex $x$ such that $\vartheta(P,x)$ is not less than $2$ in each of the configurations sketched in Fig. \ref{Afigure.1.1}. This will be stated by the two following theorems. The first of these theorems deals with the configuration shown in the left of the figure:

\begin{theorem}\label{Acorollary.2}
Let $U$ and $V$ be two triangulations of a convex polygon $\pi$. Further consider two clockwise oriented boundary edges $(a,b)$ and $(c,d)$ on $\pi$. If $\{a,\mathrm{lk}_U(\{a,b\})\}$ and $\{c,\mathrm{lk}_U(\{c,d\})\}$ are distinct and if they respectively cross edges $\{d,\mathrm{lk}_V(\{c,d\})\}$ and $\{b,\mathrm{lk}_V(\{a,b\})\}$, then $\vartheta(\{U,V\},a)$ and $\vartheta(\{U,V\},c)$ cannot both be less than $2$.
\end{theorem}
\begin{proof}
Denote vertices $\mathrm{lk}_U(\{a,b\})$, $\mathrm{lk}_V(\{a,b\})$, $\mathrm{lk}_U(\{c,d\})$, and $\mathrm{lk}_V(\{c,d\})$ by $x$, $x'$, $y$, and $y'$ respectively. Assume that edge $\{a,x\}$ crosses $\{d,y'\}$ and that edge $\{c,y\}$ crosses $\{b,x'\}$. In this case, $x$ and $x'$ are necessarily distinct. Otherwise, by assumption, $\{a,x'\}$ and $\{d,y'\}$ would be two crossing edges contained in triangulation $V$. One obtains that $y$ and $y'$ are distinct from a similar argument.

Since $x$ and $x'$ are distinct, at least one flip is incident to edge $\{a,b\}$ along any path between triangulations $U$ and $V$. Similarly, as $y$ is distinct from $y'$, some flip must be incident to edge $\{c,d\}$ along any such path.

Now assume that $\vartheta(\{U,V\},a)$ is not greater than $1$ and consider a geodesic path $(T_i)_{0\leq{i}\leq{k}}$ from $U$ to $V$. In this case, there is a unique integer $j\in\{1, ..., k\}$ such that the $j$-th flip along path $(T_i)_{0\leq{i}\leq{k}}$ is incident to $\{a,b\}$. In particular, $T_i$ contains edge $\{a,x\}$ when $i<j$ and edge $\{b,x'\}$ when $i\geq{j}$. It will be shown indirectly that at least two flips are incident to $\{c,d\}$ along $(T_i)_{0\leq{i}\leq{k}}$. Indeed, assume that there is a unique such flip, say the $l$-th one. It follows that $T_i$ contains edge $\{c,y\}$ when $i<l$ and edge $\{d,y'\}$ when $i\geq{l}$. As a consequence, $j\leq{l}$ (otherwise, $\{a,x\}$ and $\{d,y'\}$ would be crossing edges contained in $T_l$) and $l\leq{j}$ (otherwise, $\{c,y\}$ and $\{b,x'\}$ would be crossing edges contained in $T_j$), which proves that $j$ and $l$ are equal.

It has been shown that $\{a,x\}$ and $\{c,y\}$ both belong to $T_{j-1}$ and that $\{b,x'\}$ and $\{d,y'\}$ both belong to $T_j$. Hence, $\{a,x\}$ and $\{c,y\}$ are two edges of $T_{j-1}$, and each of them crosses an edge of $T_j$. As $T_j$ is obtained from $T_{j-1}$ by a flip, $\{a,x\}$ and $\{c,y\}$ must be identical. However, by assumption these edges are distinct.

As a consequence, at least two flips are incident to $\{c,d\}$ along path $(T_i)_{0\leq{i}\leq{k}}$.
\begin{figure}
\begin{centering}
\includegraphics{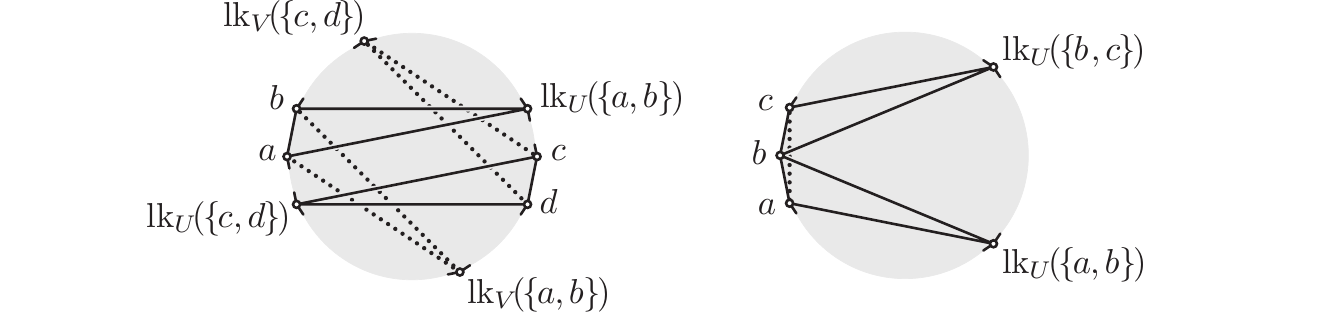}
\caption{The configurations dealt with in Theorem \ref{Acorollary.2} (left) and in Theorem \ref{Acorollary.3} (right). In both cases, solid edges belong to triangulation $U$, and dotted edges to triangulation $V$.}\label{Afigure.1.1}
\end{centering}
\end{figure}
It immediately follows that $\vartheta(\{U,V\},c)$ cannot be smaller than $2$. \qed
\end{proof}

Under stronger conditions on $U$ and $V$, the statement of Theorem \ref{Acorollary.2} can be simplified. The resulting theorem will be invoked more frequently than Theorem \ref{Acorollary.2} in the remainder of the article. It deals with the configuration shown in the right of Fig. \ref{Afigure.1.1}.

\begin{theorem}\label{Acorollary.3}
Let $U$ and $V$ be two triangulations of a convex polygon $\pi$. Further consider two clockwise oriented boundary edges $(a,b)$ and $(b,c)$ on $\pi$. If vertices $a$, $c$, $\mathrm{lk}_U(\{a,b\})$, and $\mathrm{lk}_U(\{b,c\})$ are pairwise distinct and if $\{a,c\}$ is an edge of $V$, then $\vartheta(\{U,V\},a)$ and $\vartheta(\{U,V\},b)$ cannot both be less than $2$.
\end{theorem}
\begin{proof}
Assume that $\{a,c\}$ is an edge of $V$ and that vertices $a$, $c$, $x$, and $y$ are pairwise distinct, where $x$ and $y$ denote the respective links of edges $\{a,b\}$ and $\{b,c\}$ in $U$. It immediately follows that the link of $\{a,b\}$ in $V$ is $c$ and that edge $\{a,b\}$ has distinct links in triangulations $U$ and $V$. As a consequence, at least one flip is incident to $\{a,b\}$ along any path between these triangulations. 

Assume that  $\vartheta(\{U,V\},a)\leq1$ and consider a geodesic path $(T_i)_{0\leq{i}\leq{k}}$ from $U$ to $V$. In this case, there is a unique integer $j\in\{1, ..., k\}$ so that the $j$-th flip along path $(T_i)_{0\leq{i}\leq{k}}$ is incident to $\{a,b\}$. It follows that the links of $\{a,b\}$ in triangulations $T_{j-1}$ and $T_j$ are respectively $x$ and $c$. In particular $\{b,x\}$ is an edge of $T_{j-1}$ and $\{a,c\}$ is an edge of $T_j$. Since $a$, $c$, $x$, and $y$ are pairwise distinct vertices of $\pi$, and since $\{a,b\}$ and $\{b,c\}$ are boundary edges on $\pi$, edges $\{a,c\}$ and $\{b,x\}$ are crossing. As a consequence, $T_j$ is necessarily obtained replacing $\{b,x\}$ by $\{a,c\}$ within $T_{j-1}$. Hence, the links of $\{b,c\}$ in triangulations $T_{j-1}$ and $T_j$ are respectively $x$ and $a$.

Since $x$, $a$, and $y$, are distinct, there are at least two flips incident to edge $\{b,c\}$ along path $(T_i)_{0\leq{i}\leq{k}}$: the $j$-th one, and another one that changes the link of $\{b,c\}$ from vertex $y$ to vertex $x$ (this flip necessarily takes place between triangulations $T_0$ and $T_{j-1}$). By definition, $\vartheta(\{U,V\},b)$ is therefore not less than $2$. \qed
\end{proof}

Theorem \ref{Acorollary.3} can be generalized to sequences of deletions. The following corollary provides such a generalization for the configuration sketched in the left of Fig. \ref{Afigure.1.101}:

\begin{corollary}\label{Acorollary.4}
Let $U$ and $V$ be two triangulations of a convex polygon $\pi$. Further consider four clockwise oriented boundary edges $(a,b)$, $(b,c)$, $(c,d)$, and $(d,e)$ on $\pi$. If vertices $a$, $b$, $e$, $\mathrm{lk}_U(\{b,c\})$, $\mathrm{lk}_U(\{c,d\})$, and $\mathrm{lk}_U(\{d,e\})$ are pairwise distinct and if $\{a,e\}$, $\{b,e\}$, and $\{c,e\}$ are edges of $V$, then the following inequality holds:
$$
\delta(\{U,V\})\geq\delta(\{U,V\}{\contract}b{\contract}c{\contract}d)+5\mbox{.}
$$
\end{corollary}
\begin{proof}
Assume that $a$, $b$, $e$, $\mathrm{lk}_U(\{b,c\})$, $\mathrm{lk}_U(\{c,d\})$, and $\mathrm{lk}_U(\{d,e\})$ are pairwise distinct. Observe that, in this case, vertices $\mathrm{lk}_U(\{b,c\})$, $\mathrm{lk}_U(\{c,d\})$, and $\mathrm{lk}_U(\{d,e\})$ are also necessarily distinct from $c$ and from $d$: one of these vertices would otherwise be equal to $b$ or to $e$. In particular, triangulation $U$ contains the edges depicted using solid lines in the left of Fig. \ref{Afigure.1.101}. Further assume that $\{a,e\}$, $\{b,e\}$, and $\{c,e\}$ belong to $V$. These edges are shown as dotted lines in the left of Fig. \ref{Afigure.1.101}.

As $V$ contains $\{c,e\}$ and as vertices $c$, $e$, $\mathrm{lk}_U(c)$, and $\mathrm{lk}_U(d)$ are pairwise distinct, Theorem \ref{Acorollary.3} provides a vertex $x\in\{c,d\}$ such that:
\begin{equation}\label{Acorollary.4.eq.1}
\vartheta(\{U,V\},x)\geq2\mbox{.}
\end{equation}

Denote by $c'$ the vertex of $\{c,d\}$ distinct from $x$. If follows that $(c',e)$ is a clockwise oriented boundary edge on $\pi{\contract}x$. The link of edge $\{c',e\}$ in $U{\contract}x$ is either equal to $\mathrm{lk}_U(\{c,d\})$ or to $\mathrm{lk}_U(\{d,e\})$ depending on whether $x$ is equal to $c$ or to $d$, as sketched in the center of Fig. \ref{Afigure.1.101}.
\begin{figure}
\begin{centering}
\includegraphics{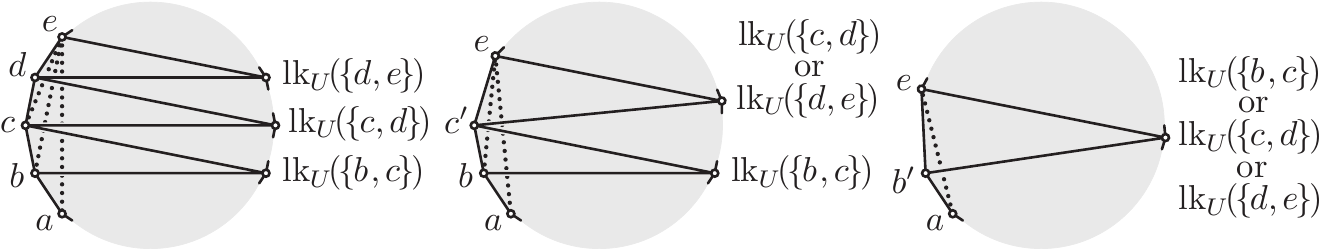}
\caption{Triangulations $U$ (left), $U{\contract}x$ (center), and $U{\contract}x{\contract}y$ (right) used in the proof of Corollary \ref{Acorollary.4}, depicted using solid lines. The dotted lines sketch triangulations $V$ (left), $V{\contract}x$ (center), and $V{\contract}x{\contract}y$ (right).}\label{Afigure.1.101}
\end{centering}
\end{figure}
In particular, vertices $b$, $e$, $\mathrm{lk}_{U{\contract}x}(\{b,c'\})$, and $\mathrm{lk}_{U{\contract}x}(\{c',e\})$ are pairwise distinct. Moreover, $\{b,e\}$ is an edge of $V{\contract}x$.

Therefore, according to Theorem \ref{Acorollary.3}, there exists $y\in\{b,c'\}$ such that:
\begin{equation}\label{Acorollary.4.eq.2}
\vartheta(\{U,V\}{\contract}x,y)\geq2\mbox{,}
\end{equation}

Denote by $b'$ the vertex of $\{b,c'\}$ distinct from $y$. In this case, $(b',e)$ is a clockwise oriented boundary edge on $\pi{\contract}x{\contract}y$. The link of $\{b',e\}$ in triangulation $U{\contract}x{\contract}y$ is equal to $\mathrm{lk}_U(\{d,e\})$, to $\mathrm{lk}_U(\{c,d\})$, or to $\mathrm{lk}_U(\{b,c\})$ depending on the values of $x$ and $y$, as sketched in the right of Fig. \ref{Afigure.1.101}. In addition, the link of this edge in $V{\contract}x{\contract}y$ is $a$. In particular, $\{b',e\}$ has distinct links in $U{\contract}x{\contract}y$ and in $V{\contract}x{\contract}y$. As a consequence:
\begin{equation}\label{Acorollary.4.eq.3}
\vartheta(\{U,V\}{\contract}x{\contract}y,b')\geq1\mbox{.}
\end{equation}

According to (\ref{Acorollary.4.eq.1}), (\ref{Acorollary.4.eq.2}), and (\ref{Acorollary.4.eq.3}), invoking Corollary \ref{Acorollary.1} three times yields:
\begin{equation}\label{Acorollary.4.eq.4}
\delta(\{U,V\})\geq\delta(\{U,V\}{\contract}x{\contract}y{\contract}b')+5\mbox{.}
\end{equation}

Finally, observe that the pair obtained deleting $b$, $c$, and $d$ from $\{U,V\}$ does not depend on the order in which these deletions are performed. In particular, pairs $\{U,V\}{\contract}x{\contract}y{\contract}b'$ and $\{U,V\}{\contract}b{\contract}c{\contract}d$ are equal. Hence, (\ref{Acorollary.4.eq.4}) is precisely the desired inequality. \qed
\end{proof}

\section{Three pairs of triangulations}
\label{Asection.6}

In this section, three pairs $A_n$, $B_n$, and $C_n$ of triangulations of a convex polygon with $n$ vertices are defined. It will be shown in the forthcoming sections that pair $A_n$ has flip distance $2n-10$ when $n$ is greater than $12$. The two other pairs will be used to state intermediate results. The two triangulations in pair $A_n$ are first described informally in the following paragraph in order to provide some preliminary intuition.

Consider the triangulations of the hexagon shown in Fig. \ref{Afigure.1}. The top left triangulation is made up of three interior edges that share a common vertex. A triangulation of this kind will be called a \emph{comb}. More generally, when a vertex $a$ of a triangulation $T$ is incident to $k\geq3$ interior edges of $T$, it will be said that $T$ admits a comb with $k$ \emph{teeth} at vertex $a$. Now consider the top right triangulation in Fig. \ref{Afigure.1}. Its edges form a simple path that alternates between left and right turns. Such a triangulation is called a \emph{zigzag}. As combs, zigzags can be arbitrarily large (see Fig. \ref{Afigure.2}). It should be noted that (the duals of) zigzags and combs have already been used in the formalism of binary trees to investigate the problem at hand \cite{Deh10}. The triangulations in pair $A_n$ are made up of a zigzag with small combs attached at both ends (see Fig. \ref{Afigure.3}). Their duals are not unrelated to the binary trees whose rotation distance is conjectured in \cite{Deh10} to be maximal. However, this conjecture will be disproved in the last section.

Let $\pi$ be a convex polygon with $n$ vertices labeled clockwise from $0$ to $n-1$. The vertices of a convex polygon labeled this way will be referred to hereafter using any integer congruent to their label modulo $n$. In addition, a vertex with label $k+\lfloor{n/2}\rfloor$ will also be denoted by $\bar{k}$. The triangulation of $\pi$ whose interior edges form a zigzag that starts at vertex $2$ as shown in Fig. \ref{Afigure.2} will be called $Z_n$.
\begin{figure}[b]
\begin{centering}
\includegraphics{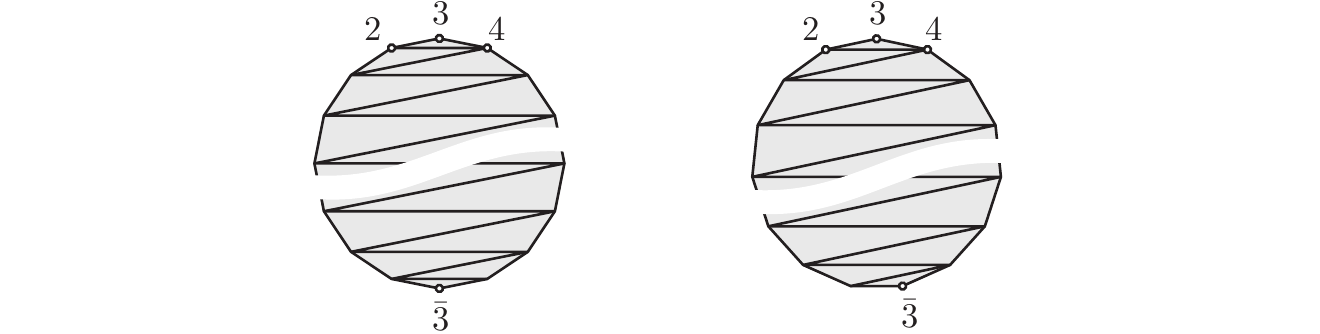}
\caption{Triangulation $Z_n$ when $n$ is even (left) and when $n$ is odd (right).}\label{Afigure.2}
\end{centering}
\end{figure}
As can be seen, the way this zigzag ends near vertex $\bar{3}$ depends on whether $n$ is even or odd.

\subsection{A first pair of triangulations}

Let $n$ be an integer not less than $3$. Call $A_n^-$ the triangulation obtained by successively deleting vertices $0$, $1$, $\bar{0}$, and $\bar{1}$ from $Z_{n+4}$:
$$
A_n^-=Z_{n+4}{\contract}0{\contract}1{\contract}\bar{0}{\contract}\bar{1}\mbox{.}
$$

According to this definition, $A_n^-$ has $n$ vertices, whose labels do not form a set of consecutive integers. For this reason, the vertices of $A_n^-$ are relabeled clockwise from $0$ to $n-1$ with the requirement that vertex $2$ keeps its label. Call $A_n^+$ the triangulation obtained by successively deleting vertices $4$, $5$, $\bar{4}$, and $\bar{5}$ from $Z_{n+4}$:
$$
A_n^+=Z_{n+4}{\contract}4{\contract}5{\contract}\bar{4}{\contract}\bar{5}\mbox{.}
$$

Again, according to this definition, triangulation $A_n^+$ has $n$ vertices, whose labels do not form a set of consecutive integers. The vertices of $A_n^+$ are therefore relabeled clockwise from $0$ to $n-1$ in such a way that vertex $2$ is relabeled $1$. In addition, each vertex of $A_n^+$ is displaced to the vertex of $A_n^-$ with the same label in order to obtain two triangulations of the same polygon.

Triangulations $A_n^-$ and $A_n^+$ are depicted in Fig. \ref{Afigure.3} when $n$ is greater than $8$. One can see that, for such values of $n$, the first two deletions carried out in $Z_{n+4}$ produce a comb with three teeth at vertex $2$ in $A_n^-$ and a comb with four teeth at vertex $3$ in $A_n^+$. The other two deletions introduce a comb at vertex $\bar{2}$ in $A_n^-$ and a comb at vertex $\bar{3}$ in $A_n^+$ whose numbers of teeth (three or four) depend on the parity of $n$. When $n$ is greater than $10$, the two combs contained in each of these triangulations are connected by a zigzag. If $n$ is equal to $10$, this zigzag shrinks to a single edge in $A_n^-$ and vanishes from $A_n^+$. If $n$ is equal to $9$, the zigzag disappears from both triangulations and in each of them the two combs have a common tooth. When $n$ is less than $9$, the combs in $A_n^-$ and $A_n^+$ lose teeth or even disappear, and these triangulations cannot be represented as in Fig. \ref{Afigure.3}. If $n=3$, then $A_n^-$ and $A_n^+$ both shrink to a single triangle.
\begin{figure}
\begin{centering}
\includegraphics{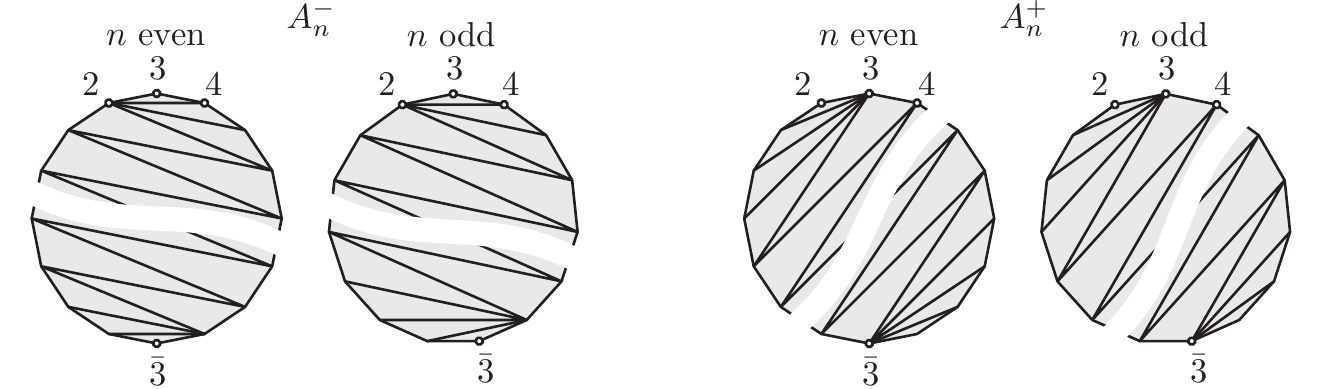}
\caption{Triangulations $A_n^-$ (left) and $A_n^+$ (right), depicted when $n$ is greater than $8$.}\label{Afigure.3}
\end{centering}
\end{figure}
Triangulations $A_n^-$ and $A_n^+$ are depicted in Fig. \ref{Afigure.4} when $n$ ranges from $4$ to $8$.

For any integer $n$ greater than $2$, denote:
$$
A_n=\{A_n^-,A_n^+\}\mbox{.}
$$

While an accurate depiction of $A_n^-$ and $A_n^+$ requires a disjunction on the value of $n$, some properties of these triangulations hold for all values of $n$. For instance, one can check using Fig. \ref{Afigure.3} and Fig. \ref{Afigure.4} that deleting vertex $1$ from pair $A_n$ results in a pair of triangulations isomorphic to $A_{n-1}$ whenever $n$ is greater than $3$. In particular, the following proposition holds:

\begin{proposition}\label{Aproposition.6.5}
For every integer $n$ greater than $3$, $\delta(A_{n-1})=\delta(A_n{\contract}1)$.
\end{proposition}
\begin{proof}
Let $n$ be an integer greater than $3$. Consider the bijection from the vertex set of $A_n^+{\contract}1$ to the vertex set of $A_{n-1}^-$ that first relabels the vertices of $A_n^+{\contract}1$ counterclockwise from $0$ to $n-2$ in such a way that vertex $2$ is relabeled $3$ and then maps each of these vertices to the vertex of $A_{n-1}^-$ with the same label. This bijection sends triangulations $A_n^+{\contract}1$ and $A_n^-{\contract}1$ to $A_{n-1}^-$ and $A_{n-1}^+$ respectively. Hence, pairs $A_{n-1}$ and $A_n{\contract}1$ are isomorphic and, therefore, they have the same flip distance. \qed
\end{proof}

One obtains the following first result regarding the flip distance of pair $A_n$ by invoking Proposition \ref{Aproposition.6.5} together with Corollary \ref{Acorollary.1}:

\begin{theorem}\label{Atheorem.7}
Let $n$ be an integer greater than $3$. If $\vartheta(A_n,1)$ is not less than $2$, then:
$$
\delta(A_n)\geq\delta(A_{n-1})+2
$$
\end{theorem}
\begin{proof}
Assume that $\vartheta(A_n,1)$ is not less than $2$. In this case, it immediately follows from Corollary \ref{Acorollary.1} that $\delta(A_n)\geq\delta(A_n{\contract}1)+2$. Since, according to Proposition \ref{Aproposition.6.5}, $\delta(A_n{\contract}1)$ is precisely equal to $\delta(A_{n-1})$, one obtains the desired inequality. \qed
\end{proof}

Theorem \ref{Atheorem.7} is a first (small) step towards a lower bound of the form $2n+O(1)$ on the flip distance of pair $A_n$. A sequence of two deletions is now considered. Let $n$ be an integer greater than $4$. It can be checked using Fig. \ref{Afigure.3} and Fig. \ref{Afigure.4} that deleting vertex $3$ and then vertex $1$ from pair $A_n$ results in a pair of triangulations isomorphic to $A_{n-2}$.
\begin{figure}
\begin{centering}
\includegraphics{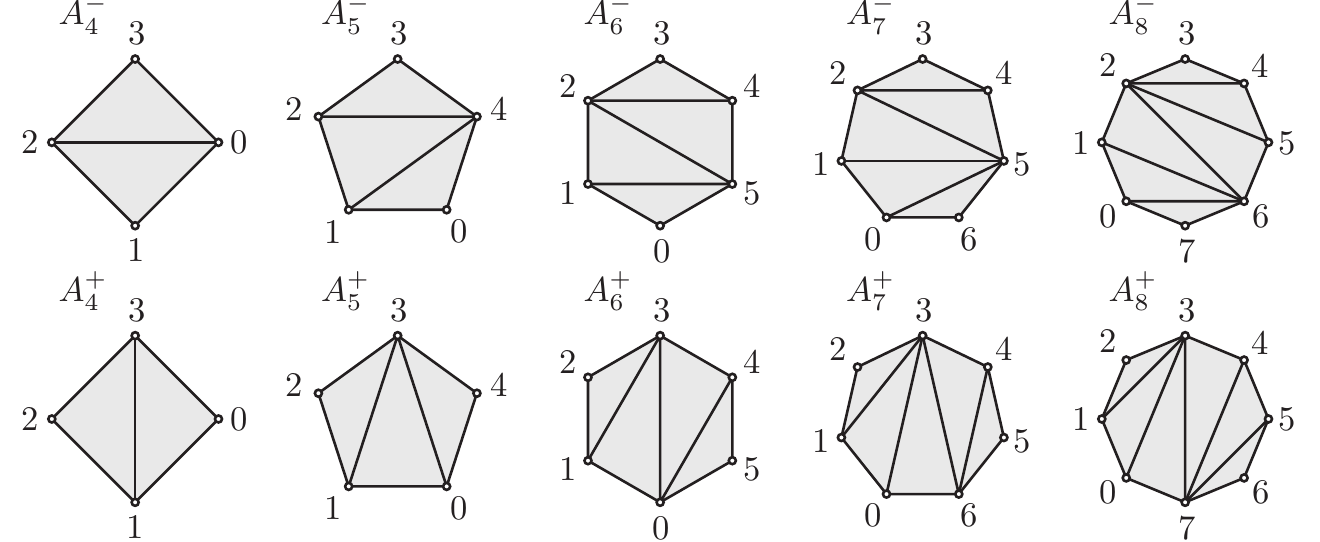}
\caption{Triangulations $A_n^-$ (top) and $A_n^+$ (bottom) when $n$ ranges from $4$ to $8$.}\label{Afigure.4}
\end{centering}
\end{figure}
The following theorem is a direct consequence of this observation:

\begin{theorem}\label{Atheorem.8}
Let $n$ be an integer greater than $5$. If $\vartheta(A_n,3)$ is not less than $3$, then:
$$
\delta(A_n)\geq\delta(A_{n-2})+4\mbox{.}
$$
\end{theorem}
\begin{proof}
Assume that $\vartheta(A_n,3)$ is not less than $3$. In this case, Corollary \ref{Acorollary.1} yields:
\begin{equation}\label{Atheorem.8.eq.1}
\delta(A_n)\geq\delta(A_n{\contract}3)+3\mbox{.}
\end{equation}

One can see in Fig. \ref{Afigure.3} and in Fig. \ref{Afigure.4} that the link of edge $\{1,2\}$ in $A_n^-{\contract}3$ is vertex $5$ when $6\leq{n}\leq7$ and vertex $6$ when $n$ is greater than $7$. Moreover, the link of this edge in $A_n^+{\contract}3$ is always vertex $4$. In particular, $\{1,2\}$ has distinct links in $A_n^-{\contract}3$ and in $A_n^+{\contract}3$ whenever $n$ is greater than $5$. It immediately follows that $\vartheta(A_n{\contract}3,1)$ is positive and according to Corollary \ref{Acorollary.1}, the following inequality holds:
\begin{equation}\label{Atheorem.8.eq.2}
\delta(A_n{\contract}3)\geq\delta(A_n{\contract}3{\contract}1)+1\mbox{.}
\end{equation}

As mentioned above, pairs $A_{n-2}$ and $A_n{\contract}3{\contract}1$ are isomorphic. The desired inequality is therefore obtained combining (\ref{Atheorem.8.eq.1}) and (\ref{Atheorem.8.eq.2}). \qed
\end{proof}

\subsection{Two other pairs of triangulations}

If the conditions required by Theorem \ref{Atheorem.7} and Theorem \ref{Atheorem.8} never fail together, then the flip distance of $A_n$ admits a lower bound of the form $2n+O(1)$. However, no argument is available to make sure that one of these conditions is always met. Geodesic paths between triangulations $A_n^-$ and $A_n^+$ will therefore be studied under the complementary assumption that $\vartheta(A_n,1)\leq1$ and $\vartheta(A_n,3)\leq2$. This study relies on two auxiliary pairs of triangulations $B_n$ and $C_n$ that are now constructed. This construction is motivated by the structure of the proof of the next theorem.

Consider an integer $n$ greater than $6$. In can be seen in Fig. \ref{Afigure.3} and in Fig. \ref{Afigure.4} that, for any such value of $n$, $\{2,4\}$, $\{2,5\}$, and $\{2,6\}$ are interior edges of $A_{n+1}^-$. Call $T$ the triangulation obtained by flipping $\{2,5\}$, $\{2,4\}$, and $\{2,6\}$ in this order in $A_{n+1}^-$:
$$
T=A_{n+1}^-\mathord{/}\{2,5\}\mathord{/}\{2,4\}\mathord{/}\{2,6\}\mbox{.}
$$

Note that the third flip replaces $\{2,6\}$ by $\{1,3\}$. The latter edge therefore belongs to $T$. Further observe that $\{1,3\}$ is also an edge of $A_{n+1}^+$. Hence, calling $E$ the set whose elements are edges $\{1,2\}$ and $\{2,3\}$, two triangulations $B_n^-$ and $B_n^+$ of a polygon with $n$ vertices can be defined as follows:
$$
B_n^-=T\mathord{\setminus}{E}\mbox{, and }B_n^+=A_{n+1}^+\mathord{\setminus}{E}\mbox{.}
$$

Note that vertex $2$ has been removed from these two triangulations. In order to keep vertex labels consecutive, vertices are relabeled clockwise in such a way that vertex $3$ keeps its label. Triangulations $B_n^-$ and $B_n^+$ are depicted in Fig. \ref{Afigure.5} when $n$ is greater than $8$ and in Fig. \ref{Afigure.10} when $7\leq{n}\leq11$. In these figures, the dotted edges depict a variation of $B_n^-$ that will be introduced below. One can see that, when $n$ is greater than $8$, the flips carried out for the construction of $B_n^-$ result in a comb with three teeth at vertex $6$. The other comb, at vertex $\bar{2}$ when $n$ is even and at vertex $\bar{3}$ when $n$ is odd, is inherited from $A_{n+1}^-$. These two combs are connected by a zigzag when $n>10$, are adjacent when $9\leq{n}\leq10$, and merge into a single comb when $7\leq{n}\leq8$.

Observe that $\{4,6\}$ is an interior edge of $B_n^-$. The triangulation obtained by flipping this edge in $B_n^-$ will be denoted by $C_n^-$. Hence, triangulation $C_n^-$ differs from $B_n^-$ by exactly one edge depicted as a dotted line in the figures.
\begin{figure}
\begin{centering}
\includegraphics{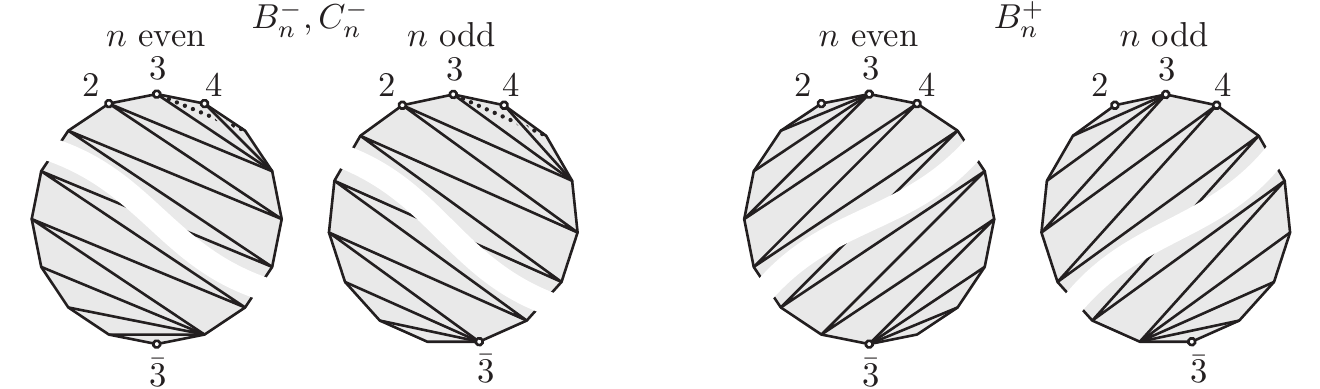}
\caption{Triangulations $B_n^-$ (left) and $B_n^+$ (right), depicted in solid lines when $n$ is greater than $8$. Triangulation $C_n^-$, obtained by flipping edge $\{4,6\}$ in $B_n^-$, is sketched using dotted lines (left).}\label{Afigure.5}
\end{centering}
\end{figure}
Observe that $C_n^-$ can also be built by flipping edges $\{2,4\}$, $\{2,5\}$, and $\{2,6\}$ in this order in $A_n^-$, removing the elements of $E$ from the resulting triangulation, and relabeling its vertices.

For any integer $n$ greater than $6$, consider the following pairs of triangulations:
$$
B_n=\{B_n^-,B_n^+\}\mbox{ and }C_n=\{C_n^-,B_n^+\}\mbox{.}
$$

If the conditions required by Theorem \ref{Atheorem.7} and Theorem \ref{Atheorem.8} fail together, the following theorem allows to bound the flip distance of $A_n$ below using $B_{n-1}$ or $C_{n-1}$: 

\begin{theorem}\label{Atheorem.9}
Let $n$ be an integer greater than $7$. If $\vartheta(A_n,1)$ and $\vartheta(A_n,3)$ are respectively not greater than $1$ and not greater than $2$, then there exists a pair $P\in\{B_{n-1},C_{n-1}\}$ so that $\delta(P)$ is equal to $\delta(A_n)-3$ and $\vartheta(P,3)$ is not greater than $1$.
\end{theorem}
\begin{proof}
Assume that $\vartheta(A_n,1)\leq1$ and that $\vartheta(A_n,3)\leq2$. Consider a geodesic path $(T_i)_{0\leq{i}\leq{k}}$ from $A_n^-$ to $A_n^+$. Since $n$ is greater than $7$, the link of edge $\{1,2\}$ in triangulation $A_n^-$ is vertex $6$ and its link in $A_n^+$ is vertex $3$ (see Fig. \ref{Afigure.3} when $n>8$ and Fig. \ref{Afigure.4} when $n=8$). These links being distinct, there exists an integer $j\in\{1, ..., k\}$ so that the $j$-th flip along path $(T_i)_{0\leq{i}\leq{k}}$ is incident to edge $\{1,2\}$. Since $\vartheta(A_n,1)$ is not greater than $1$ there is no other such flip along this path. Hence, the link of edge $\{1,2\}$ in $T_i$ is vertex $6$ when $i<j$ and vertex $3$ when $i\geq{j}$. It follows that edges $\{2,6\}$ and $\{1,3\}$ respectively belong to triangulations $T_{j-1}$ and $T_j$. As these edges are crossing, $T_j$ is necessarily obtained replacing $\{2,6\}$ within $T_{j-1}$ by $\{1,3\}$. In particular, these triangulations both contain all the boundary edges of quadrilateral $\{1,2,3,6\}$.

As a first consequence, $\{3,6\}$ is an edge of $T_j$. Therefore, all the boundary edges of quadrilateral $\{3,4,5,6\}$ belong to $T_j$, and so does one of its diagonals (i.e. $\{3,5\}$ or $\{4,6\}$). Denote these diagonals by $\varepsilon$ and $\varepsilon'$ with the convention that $\varepsilon$ belongs to $T_j$. It can be seen in Fig. \ref{Afigure.3} and in Fig. \ref{Afigure.4} that each of these diagonals can be introduced in triangulation $A_n^-$ by a flip: flipping $\{2,4\}$ in this triangulation introduces edge $\{3,5\}$ and flipping $\{2,5\}$ instead introduces edge $\{4,6\}$. Denote edges $\{2,4\}$ and $\{2,5\}$ by $\varsigma$ and $\varsigma'$ with the convention that flipping $\varsigma$ in $A_n^-$ introduces $\varepsilon$.

It follows from this construction that edges $\varepsilon$, $\{3,6\}$, and $\{1,3\}$ can be introduced in this order in triangulation $A_n^-$ by successively flipping $\varsigma$, $\varsigma'$, and $\{2,6\}$. This can be checked using Fig. \ref{Afigure.3} when $n$ is greater than $8$ and Fig. \ref{Afigure.4} when $n$ is equal to $8$. Now recall that $\varepsilon$, $\{3,6\}$, and $\{1,3\}$ are contained in $T_j$. Hence, according to Theorem \ref{Atheorem.3}, it can be assumed without loss of generality that $\varsigma$, $\varsigma'$, and $\{2,6\}$ are the first three edges flipped in this order along path $(T_i)_{0\leq{i}\leq{k}}$.

Observe that, in this case, $\{1,3\}$ belongs to both $T_3$ and $A_n^+$. Therefore, according to Lemma \ref{Alemma.3}, this edge belongs to $T_i$ whenever $i\geq3$. In particular, calling $E$ the set whose elements are $\{1,2\}$ and $\{2,3\}$, triangulations $T_3\mathord{\setminus}{E}$, ..., $T_k\mathord{\setminus}{E}$ form a geodesic path from $T_3\mathord{\setminus}{E}$ to $A_n^+\mathord{\setminus}{E}$. Note that this path inherits its geodesicity from path $(T_i)_{0\leq{i}\leq{k}}$.
\begin{figure}
\begin{centering}
\includegraphics{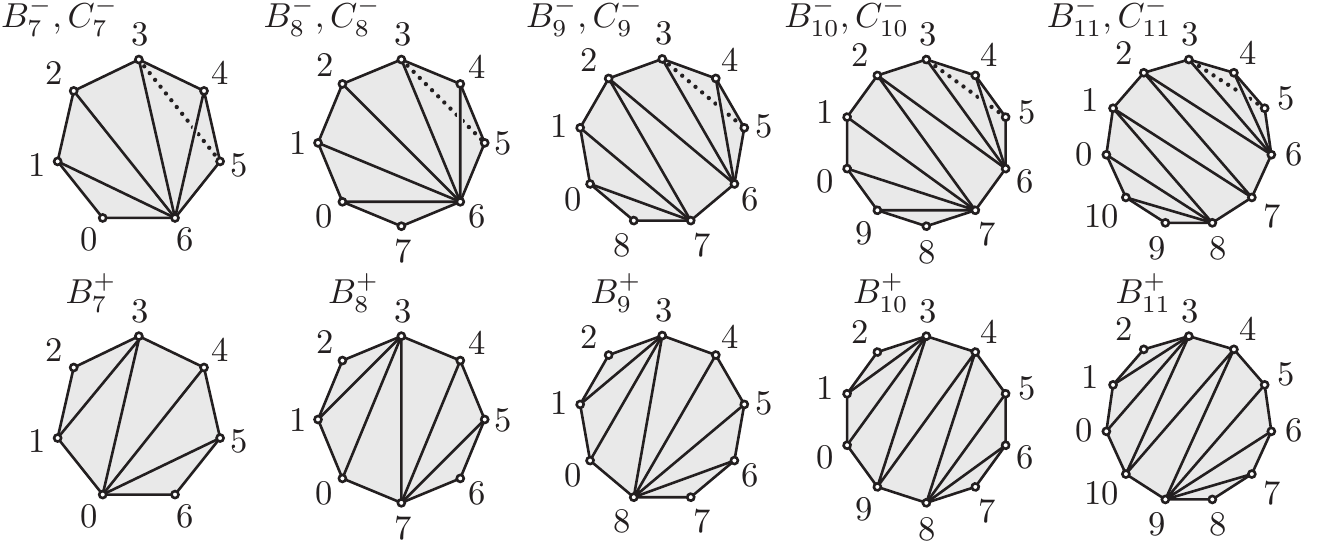}
\caption{Triangulations $B_n^-$ (top) and $B_n^+$ (bottom), depicted in solid lines when $7\leq{n}\leq11$. Triangulation $C_n^-$, obtained by flipping edge $\{4,6\}$ in $B_n^-$ is sketched in the top of the figure using dotted lines.}\label{Afigure.10}
\end{centering}
\end{figure}
As a consequence, pair $P=\{T_3\mathord{\setminus}{E},A_n^+\mathord{\setminus}{E}\}$ has flip distance $k-3$, that is:
\begin{equation}\label{Atheorem.9.eq.1}
\delta(A_n)=\delta(P)+3\mbox{.}
\end{equation}

Further note that one of the first three flips along $(T_i)_{0\leq{i}\leq{k}}$ is incident to $\{3,4\}$. This flip is the first one along the path if $\varepsilon$ is equal to $\{3,5\}$ and the second one if $\varepsilon$ is equal to $\{4,6\}$.
Since the last $k-3$ flips along path $(T_i)_{0\leq{i}\leq{k}}$ can be prescribed to be the ones of any geodesic path from $T_3\mathord{\setminus}{E}$ to $A_n^+\mathord{\setminus}{E}$, the following inequality holds:
$$
\vartheta(A_n,3)\geq\vartheta(P,3)+1\mbox{.}
$$

As $\vartheta(A_n,3)$ is not greater than $2$, one obtains:
\begin{equation}\label{Atheorem.9.eq.2}
\vartheta(P,3)\leq1\mbox{.}
\end{equation}

Finally, relabel the vertices of triangulations $T_3\mathord{\setminus}{E}$ and $A_n^+\mathord{\setminus}{E}$ clockwise in such a way that vertex $3$ keeps its label. Observe that this relabeling does not affect (\ref{Atheorem.9.eq.1}) nor (\ref{Atheorem.9.eq.2}). Moreover, by construction, $A_n^+\mathord{\setminus}{E}$ is then precisely equal to $B_{n-1}^+$, and:
$$
T_3\mathord{\setminus}{E}=
\left\{
\begin{array}{ll}
B_{n-1}^- & \mbox{ if }\varepsilon=\{4,6\}\mbox{,}\\
C_{n-1}^- & \mbox{ if }\varepsilon=\{3,5\}\mbox{.}\\
\end{array}
\right.
$$

It follows that $P$ is either equal to $B_{n-1}$ or to $C_{n-1}$, and the proof is complete. \qed
\end{proof}

\section{A recursive lower bound on $\delta(A_n)$}
\label{Asection.7}

The conditions on $\vartheta(A_n,1)$ and $\vartheta(A_n,3)$ in the statements of Theorem \ref{Atheorem.7}, Theorem \ref{Atheorem.8}, and Theorem \ref{Atheorem.9} are complementary. Hence, the inequalities provided by these theorems can be combined into the following lower bound on the flip distance of pair $A_n$:
\begin{proposition}\label{Aproposition.7}
For any integer $n$ greater than $7$,
$$
\delta(A_n)\geq\min(\{\delta(A_{n-1})+2,\delta(A_{n-2})+4,\delta(B_{n-1})+3,\delta(C_{n-1})+3\})\mbox{.}
$$
\end{proposition}

The purpose of this section is to improve Proposition \ref{Aproposition.7} into a recursive lower bound on $\delta(A_n)$, achieved bounding $\delta(B_n)$ and $\delta(C_n)$ below by $\delta(A_{n-5})+9$ and $\delta(A_{n-4})+7$ respectively. These two bounds will be obtained when $n$ is greater than $11$ under the respective conditions that $\vartheta(B_n,3)$ and $\vartheta(C_n,3)$ are not greater than $1$. Note that, according to the statement of Theorem \ref{Atheorem.9}, these conditions are not restrictive.

The bound on $\delta(B_n)$ will be found using a sequence of five deletions. The next lemma corresponds to the first two deletions:

\begin{lemma}\label{Alemma.7}
Let $n$ be an integer greater than $11$. If $\vartheta(B_n,3)\leq1$, then:
$$
\delta(B_n)\geq\delta(B_n{\contract}4{\contract}5)+4\mbox{.}
$$
\end{lemma}
\begin{proof}
Assume that $\vartheta(B_n,3)\leq1$. Let $(T_i)_{0\leq{i}\leq{k}}$ be a geodesic path from $B_n^-$ to $B_n^+$. As $n$ is greater than $11$, edge $\{3,4\}$ admits distinct links in triangulations $B_n^-$ and $B_n^+$: it can be seen in Fig. \ref{Afigure.5} that these links are respectively vertex $6$ and vertex $n-1$. Hence, there exists $j\in\{1, ..., k\}$ so that the $j$-th flip along path $(T_i)_{0\leq{i}\leq{k}}$ is incident to edge $\{3,4\}$. Since $\vartheta(B_n,3)\leq1$, there is no other such flip along this path. Hence, edges $\{3,6\}$ and $\{4,n-1\}$ respectively belong to triangulations $T_{j-1}$ and $T_j$. As these edges are crossing, triangulation $T_j$ is necessarily obtained from $T_{j-1}$ by replacing $\{3,6\}$ with $\{4,n-1\}$. In particular, $T_j$ contains edge $\{4,n-1\}$ and all the boundary edges of quadrilateral $\{3,4,6,n-1\}$ as shown in the left of Fig. \ref{Afigure.75}. Denote:
$$
P=\{B_n^-,T_j\}\mbox{ and }Q=\{T_j,B_n^+\}\mbox{.}
$$

According to the triangle inequality:
\begin{equation}\label{Alemma.7.eq.1}
\delta(P{\contract}4{\contract}5)+\delta(Q{\contract}4{\contract}5)\geq\delta(B_n{\contract}4{\contract}5)\mbox{.}
\end{equation}

The rest of the proof consists in finding upper bounds on the flip distance of pairs $P{\contract}4{\contract}5$ and $Q{\contract}4{\contract}5$. One can see in Fig. \ref{Afigure.5} that the link of $\{5,6\}$ in triangulation $B_n^-{\contract}4$ is vertex $3$. Moreover, as shown in Fig. \ref{Afigure.75} the link of this edge in $T_j{\contract}4$ is vertex $n-1$. These two links being distinct, $\vartheta(P{\contract}4,5)$ is necessarily positive. As in addition, $\vartheta(P,4)$ is non-negative, invoking Corollary \ref{Acorollary.1} twice yields:
\begin{equation}\label{Alemma.7.eq.2}
\delta(P)\geq\delta(P{\contract}4{\contract}5)+1\mbox{.}
\end{equation}

Now consider triangulation $B_n^+$, depicted in the right of Fig. \ref{Afigure.5} and recall that $\bar{3}$ stands for $\lfloor{n/2}\rfloor+3$. In particular, the figure shows that $B_n^+$ has a comb at vertex $\lceil{n/2}\rceil+3$. As $n$ is greater than $11$, the following inequality holds:
$$
\lceil{n/2}\rceil+3\leq{n-3}\mbox{.}
$$

One can see in Fig. \ref{Afigure.5} that, in this case, edges $\{4,5\}$ and $\{5,6\}$ cannot both be placed between two teeth of the comb at vertex $\lceil{n/2}\rceil+3$ in $B_n^+$.
\begin{figure}
\begin{centering}
\includegraphics{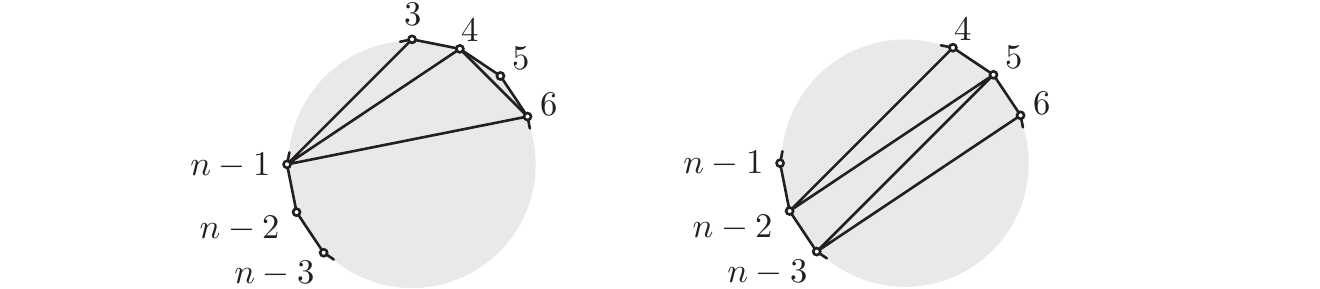}
\caption{Partial sketches of triangulations $T_j$ (left) and $B_n^+$ (right) used in the proof of Lemma \ref{Alemma.7}.}\label{Afigure.75}
\end{centering}
\end{figure}
As a consequence, the links of these edges in $B_n^+$ are respectively vertex $n-2$ and vertex $n-3$, as sketched in the right of Fig. \ref{Afigure.75}. Since vertices $4$, $6$, $n-3$, and $n-2$ are pairwise distinct, and since $\{4,6\}$ belongs to $T_j$, one can invoke Theorem \ref{Acorollary.3} with $U=B_n^+$, $V=T_j$, and $a=4$ providing a vertex $x\in\{4,5\}$ so that $\vartheta(Q,x)\geq2$. 

Now denote by $y$ the vertex of $\{4,5\}$ distinct from $x$. In this case, $\{y,6\}$ is a boundary edge of triangulations $T_j{\contract}x$ and $B_n^+{\contract}x$. It can be seen in Fig. \ref{Afigure.75} that the links of $\{y,6\}$ in $T_j{\contract}x$ and in $B_n^+{\contract}x$ are respectively vertex $n-1$ and vertex $n-i$ with $i\in\{2,3\}$. These links being distinct, $\vartheta(Q{\contract}x,y)$ is necessarily positive, and invoking Corollary \ref{Acorollary.1} twice provides the following inequality:
\begin{equation}\label{Alemma.7.eq.3}
\delta(Q)\geq\delta(Q{\contract}x{\contract}y)+3\mbox{.}
\end{equation}

Finally, observe that the pair obtained by deleting vertices $4$ and $5$ from $Q$ does not depend on the order in which these vertices are deleted. As a consequence, inequality (\ref{Alemma.7.eq.3}) necessarily holds when $x=4$ and $y=5$. Since $(T_i)_{0\leq{i}\leq{k}}$ is a geodesic path, it follows from the definition of pairs $P$ and $Q$ that $\delta(B_n)=\delta(P)+\delta(Q)$. The desired result is therefore obtained by combining inequalities (\ref{Alemma.7.eq.1}) and (\ref{Alemma.7.eq.2}) with inequality (\ref{Alemma.7.eq.3}), wherein $x$ and $y$ have been set to respectively $4$ and $5$. \qed
\end{proof}

Triangulations $B_n^-{\contract}4{\contract}5$ and $B_n^+{\contract}4{\contract}5$ are shown in Fig. \ref{Afigure.7} when $n$ is greater than $11$. Note that the vertices of these triangulations are not relabeled after the deletions. In particular, vertex $6$ follows vertex $3$ clockwise and label $\bar{3}$ still stands for $\lfloor{n/2}\rfloor+3$. Note that the sketch of $B_n^+{\contract}4{\contract}5$ in the top right of the figure is accurate only when $n>13$. This triangulation is shown in the bottom of the figure when $12\leq{n}\leq15$.

One can see that deleting vertices $0$, $1$, and $2$ from triangulation $B_n^-{\contract}4{\contract}5$ creates a comb with three teeth at vertex $3$. Performing the same deletions in $B_n^+{\contract}4{\contract}5$ removes the comb at vertex $3$, leaving only two combs in the resulting triangulation. More precisely, these deletions transform pair $B_n{\contract}4{\contract}5$ into a pair of triangulations isomorphic to $A_{n-5}$. These will be the last three deletions needed to prove the following theorem:

\begin{theorem}\label{Atheorem.10}
Let $n$ be an integer greater than $11$. If $\vartheta(B_n,3)\leq1$, then:
$$
\delta(B_n)\geq\delta(A_{n-5})+9\mbox{.}
$$
\end{theorem}
\begin{proof}
Consider triangulation $B_n^-{\contract}4{\contract}5$ depicted in Fig. \ref{Afigure.7}. This triangulation has only one comb at vertex $\lceil{n/2}\rceil+2$. As $n$ is greater than $11$, then:
$$
8\leq\lceil{n/2}\rceil+2\mbox{.}
$$

Hence, as can be seen in Fig. \ref{Afigure.7}, among the three edges $\{0,1\}$, $\{1,2\}$, and $\{2,3\}$, only the first one is possibly placed between two teeth of this comb.
\begin{figure}
\begin{centering}
\includegraphics{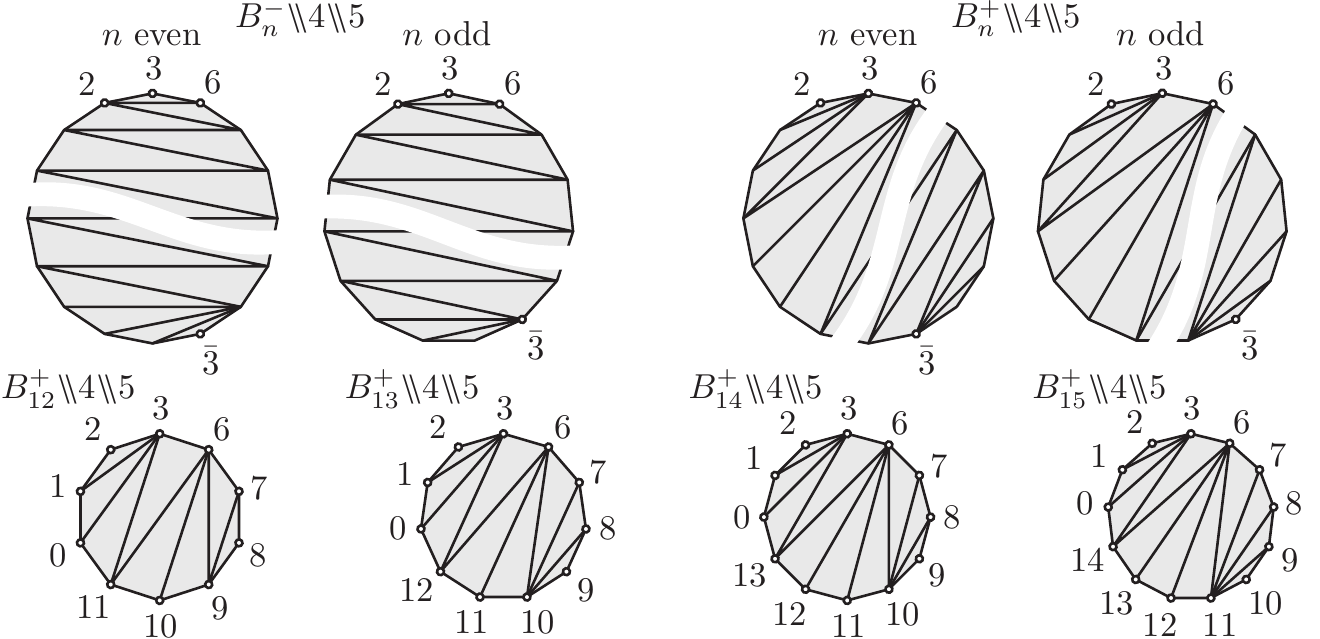}
\caption{Triangulation $B_n^-{\contract}4{\contract}5$ when $n$ is greater than $11$ (top left) and triangulation $B_n^+{\contract}4{\contract}5$ when $n$ is greater than $13$ (top right). The latter is shown in the bottom when $12\leq{n}\leq15$.}\label{Afigure.7}
\end{centering}
\end{figure}
It follows that these three edges respectively admit vertices $8$, $7$, and $6$ as their links in $B_n^-{\contract}4{\contract}5$.  Now observe that vertices $0$, $3$, $6$, $7$, $8$, and $n-1$ are pairwise distinct. Further observe that $\{0,3\}$, $\{1,3\}$, and $\{3,n-1\}$ are edges of $B_n^+{\contract}4{\contract}5$. Therefore, Corollary \ref{Acorollary.4} can be invoked with $a=n-1$, $U=B_n^-{\contract}4{\contract}5$, and $V=B_n^+{\contract}4{\contract}5$, providing the following inequality:
\begin{equation}\label{Atheorem.10.eq.1}
\delta(B_n{\contract}4{\contract}5)\geq\delta(B_n{\contract}4{\contract}5{\contract}0{\contract}1{\contract}2)+5\mbox{.}
\end{equation}

As mentioned above, pair $B_n{\contract}4{\contract}5{\contract}0{\contract}1{\contract}2$ is isomorphic to $A_{n-5}$. Hence, the desired result follows from Lemma \ref{Alemma.7} and from inequality (\ref{Atheorem.10.eq.1}). \qed
\end{proof}

The lower bound $\delta(C_n)\geq\delta(A_{n-4})+7$ is now established using a sequence of four deletions. The proof of this inequality is similar to the proof of Theorem \ref{Atheorem.10}. In particular, the last three deletions will be the same. The main difference is that a single deletion is needed instead of two in the part of the proof that corresponds to Lemma \ref{Alemma.7}. This results in a substantial simplification of the argument.

\begin{theorem}\label{Alemma.9}
Let $n$ be an integer greater than $11$. If $\vartheta(C_n,3)\leq1$, then:
$$
\delta(C_n)\geq\delta(A_{n-4})+7\mbox{.}
$$
\end{theorem}
\begin{proof}
Assume that $\vartheta(C_n,3)$ is not greater than $1$ and consider the comb at vertex $\lceil{n/2}\rceil+3$ in triangulation $B_n^+$, shown in Fig. \ref{Afigure.5}. Since $n$ is greater than $11$,
$$
\lceil{n/2}\rceil+3<{n-2}\mbox{.}
$$

Hence, as can be seen in Fig. \ref{Afigure.5}, edges $\{3,4\}$ and $\{4,5\}$ cannot be placed between two teeth of this comb. As a consequence, the links of these edges in triangulation $B_n^+$ are respectively vertex $n-1$ and vertex $n-2$. Observe that vertices $3$, $5$, $n-2$, and $n-1$ are pairwise distinct because $n$ is not less than $8$. As in addition, $\{3,5\}$ is an edge of $C_n^-$, Theorem \ref{Acorollary.3} can be invoked with $U=B_n^+$, $V=C_n^-$, and $a=3$. As a result, $\vartheta(C_n,3)$ and $\vartheta(C_n,4)$ cannot both be smaller than $2$. Since $\vartheta(C_n,3)$ is less than $2$ by assumption, this proves that $\vartheta(C_n,4)$ is at least $2$ and Corollary \ref{Acorollary.1} yields:
\begin{equation}\label{Alemma.9.eq.1}
\delta(C_n)\geq\delta(C_n{\contract}4)+2\mbox{.}
\end{equation}

Triangulations $C_n^-{\contract}4$ and $B_n^+{\contract}4$ are depicted in Fig. \ref{Afigure.9}. Note that the vertices of these triangulations have not been relabeled after the deletion.
\begin{figure}
\begin{centering}
\includegraphics{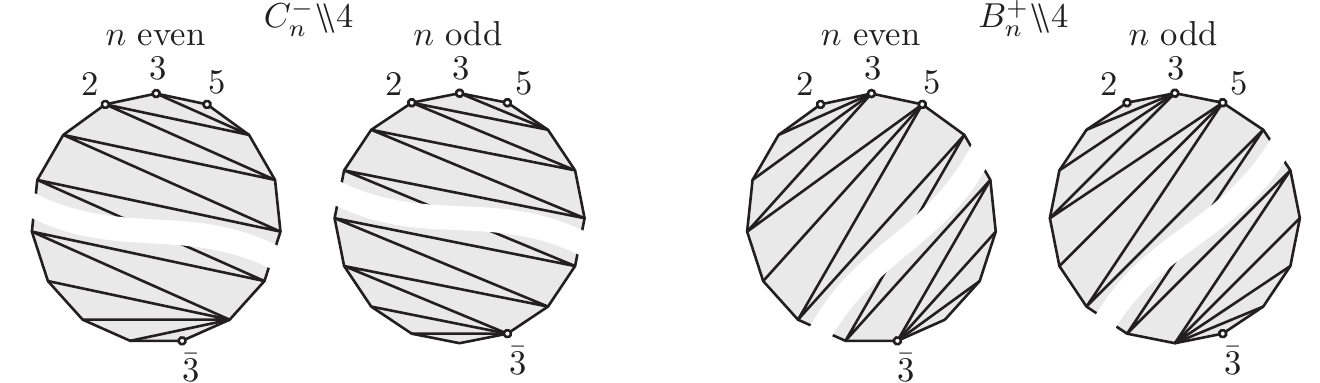}
\caption{Triangulations $C_n^-{\contract}4$ (left) and $B_n^+{\contract}4$ (right), depicted when $n$ is greater than $11$.}\label{Afigure.9}
\end{centering}
\end{figure}
One can see that the respective links of edges $\{0,1\}$, $\{1,2\}$, and $\{2,3\}$ in triangulation $C_n^-{\contract}4$ are vertices $8$, $7$, and $6$. Indeed, observe that $C_n^-{\contract}4$ has a comb at vertex $\lceil{n/2}\rceil+2$. As $n$ is greater than $11$, then $8\leq\lceil{n/2}\rceil+2$ and, among the three edges $\{0,1\}$, $\{1,2\}$, and $\{2,3\}$, only the first one may be placed between two teeth of this comb. Vertices $0$, $3$, $6$, $7$, $8$, and $n-1$ are pairwise distinct because $n$ is not less than $10$. Moreover, $\{0,3\}$, $\{1,3\}$, and $\{3,n-1\}$ are edges of $B_n^+{\contract}4$. In particular, Corollary \ref{Acorollary.4} can be invoked with $a=n-1$, $U=C_n^-{\contract}4$, and $V=B_n^+{\contract}4$, providing the following inequality:
\begin{equation}\label{Alemma.9.eq.2}
\delta(C_n{\contract}4)\geq\delta(C_n{\contract}4{\contract}0{\contract}1{\contract}2)+5\mbox{.}
\end{equation}

It can be seen in Fig. \ref{Afigure.9} that pair $C_n{\contract}4{\contract}0{\contract}1{\contract}2$ is isomorphic to $A_{n-4}$. Note in particular that the last three deletions create a comb with four teeth at vertex $3$ in triangulation $C_n^-{\contract}4$ and remove the comb at vertex $3$ from triangulation $B_n^+{\contract}4$.

The result is therefore obtained by combining (\ref{Alemma.9.eq.1}) and (\ref{Alemma.9.eq.2}). \qed
\end{proof}

The previous five theorems provide a recursive inequality on the flip distance of pair $A_n$ that holds whenever $n$ is greater than $12$. According to this inequality, there exists $i\in\{1,2,5,6\}$ so that the difference $\delta(A_n)-\delta(A_{n-i})$ is bounded below by $2i$:

\begin{theorem}\label{Atheorem.12}
For any integer $n$ greater than $12$,
$$
\delta(A_n)\geq\min(\{\delta(A_{n-1})+2,\delta(A_{n-2})+4,\delta(A_{n-5})+10,\delta(A_{n-6})+12\})\mbox{.}
$$
\end{theorem}
\begin{proof}
Let $n$ be an integer greater than $12$. If $\vartheta(A_n,1)\geq2$ or if $\vartheta(A_n,3)\geq3$, the result follows from Theorem \ref{Atheorem.7} or from Theorem \ref{Atheorem.8}. If these conditions both fail, then Theorem \ref{Atheorem.9} provides a pair of triangulations $P\in\{B_{n-1},C_{n-1}\}$ so that:
\begin{equation}\label{Atheorem.12.eq.1}
\delta(A_n)=\delta(P)+3\mbox{.}
\end{equation}

According to the same theorem, one can require that $\vartheta(P,3)\leq1$. Under this condition, and because $n-1$ is greater than $11$, Theorem \ref{Atheorem.10} and Theorem \ref{Alemma.9} yield:
\begin{equation}\label{Atheorem.12.eq.2}
\delta(P)\geq\min(\{\delta(A_{n-6})+9,\delta(A_{n-5})+7\})
\end{equation}

The result is therefore obtained combining (\ref{Atheorem.12.eq.1}) and (\ref{Atheorem.12.eq.2}). \qed
\end{proof}

\section{The diameter of associahedra}
\label{Asection.8}

It follows from Theorem \ref{Atheorem.12} that the flip distance of $A_n$ is at least $2n-O(1)$. This expression is now refined using lower bounds on the flip distances of pairs $A_3$ to $A_{12}$. The case of $A_3$, $A_4$, and $A_5$ is settled easily using Lemma \ref{Alemma.2}:

\begin{proposition}\label{Aproposition.7.9}
Pairs $A_3$, $A_4$, and $A_5$ have flip distance $0$, $1$, and $2$ respectively.
\end{proposition}
\begin{proof}
As already mentioned, $A_3^-$ and $A_3^+$ are equal to a single identical triangle. It immediately follows that their flip distance is $0$. Assume that $4\leq{n}\leq5$ and consider pair $A_n$ shown in Fig. \ref{Afigure.4}. Observe that $A_n^-$ can be transformed into $A_n^+$ by a sequence of flips that each introduce an edge of $A_n^+$. Hence, it follows from Lemma \ref{Alemma.2} that the flip distance of pair $A_n$ is equal to the number of interior edges of $A_n^+$. In other words, $A_4$ has flip distance $1$ and $A_5$ has flip distance $2$. \qed
\end{proof}

The flip distance of $A_n$ is found using more elaborate arguments when $6\leq{n}\leq8$:

\begin{proposition}\label{Aproposition.8}
Pairs $A_6$, $A_7$, and $A_8$ have flip distance $4$, $5$, and $7$ respectively.
\end{proposition}
\begin{proof}
Consider pair $A_6$, shown in Fig. \ref{Afigure.4}. Its flip distance can be obtained from Fig. \ref{Afigure.1} or alternatively from a few simple arguments. Indeed, performing a flip in triangulation $A_6^-$ will never introduce an edge of $A_6^+$. Hence the second triangulation, say $T$, in any path from $A_6^-$ to $A_6^+$ never shares an interior edge with $A_6^+$. At least $3$ other flips are therefore needed to remove all the interior edges of $T$. This shows that $\delta(A_6)$ is at least $4$. The opposite inequality is found exhibiting a path of length $4$ between triangulations $A_6^-$ and $A_6^+$ such as the following one:
$$
A_6^+=A_6^-\mathord{/}\{1,5\}\mathord{/}\{2,5\}\mathord{/}\{2,4\}\mathord{/}\{0,2\}\mbox{.}
$$

The flip distances of $A_7$ and $A_8$ can be deduced from that of $A_6$. Indeed, $\{0,1\}$ has distinct links in $A_7^-$ and $A_7^+$. Hence, $\vartheta(A_7,0)$ is positive and Corollary \ref{Acorollary.1} yields:
$$
\delta(A_7)\geq\delta(A_7{\contract}0)+1\mbox{.}
$$

It can be seen in Fig. \ref{Afigure.4} that pair $A_7{\contract}0$ is isomorphic to $A_6$. In particular, the above inequality simplifies to $\delta(A_7)\geq\delta(A_6)+1$. As a consequence, the flip distance of $A_6$ is not less than $5$. Moreover, the following $5$ flips transform $A_7^-$ into $A_7^+$:
$$
A_7^+=A_7^-\mathord{/}\{2,4\}\mathord{/}\{2,5\}\mathord{/}\{1,5\}\mathord{/}\{0,5\}\mathord{/}\{3,5\}\mbox{.}
$$

This shows that $A_7$ has flip distance $5$. Now consider pair $A_8$ shown in Fig. \ref{Afigure.4}. One can see that $\{0,6\}$ and $\{2,4\}$ are distinct edges of $A_8^-$ that respectively cross $\{5,7\}$ and $\{1,3\}$. Hence, Theorem \ref{Acorollary.2} can be invoked with $U=A_8^-$, $V=A_8^+$, $a=0$, and $c=4$. It follows that $\vartheta(A_8,0)$ and $\vartheta(A_8,4)$ cannot both be less than $2$ and Corollary \ref{Acorollary.1} yields:
$$
\delta(A_8)\geq\delta(A_8{\contract}x)+2\mbox{,}
$$
where $x$ is either equal to vertex $0$ or vertex $4$. Observe that, in both cases, pair $A_8{\contract}x$ is isomorphic to $A_7$. Therefore, the above inequality simplifies to $\delta(A_8)\geq\delta(A_7)+2$, which shows that pair $A_8$ has flip distance at least $7$. Finally, observe that:
$$
A_8^+=A_8^-\mathord{/}\{2,4\}\mathord{/}\{2,5\}\mathord{/}\{2,6\}\mathord{/}\{1,6\}\mathord{/}\{0,6\}\mathord{/}\{3,6\}\mathord{/}\{3,5\}\mbox{.}
$$

Hence, the flip distance of pair $A_8$ is exactly $7$. \qed
\end{proof}

When $9\leq{n}\leq12$, the flip distance of $A_n$ can be found using Proposition \ref{Aproposition.7}. Note that the lower bound on $\delta(A_n)$ provided by this proposition depends on $\delta(B_{n-1})$ and $\delta(C_{n-1})$. These two flip distances are now calculated. Triangulations $B_n^-$, $C_n^-$, and $B_n^+$ are shown in Fig. \ref{Afigure.10} when $n$ ranges from $7$ to $11$. The flip distances of $B_n$ and $C_n$ can be obtained when $8\leq{n}\leq11$ using results from previous sections:

\begin{proposition}\label{Aproposition.9}
If $8\leq{n}\leq11$, then $\delta(B_n)\geq2n-11$ and $\delta(C_n)\geq2n-11$.
\end{proposition}
\begin{proof}
First observe that $B_8^-$ and $C_8^-$ can both be transformed into $B_8^+$ using sequences of flips that only introduce edges of $B_8^+$. Hence, it follows from Lemma \ref{Alemma.2} that $\delta(B_8)$ and $\delta(C_8)$ are both equal to the number of interior edges of $B_8^+$, that is $5$. As $2n-11$ is precisely equal to $5$ when $n=8$, the result holds for this value of $n$.

For any $9\leq{n}\leq11$, consider a pair $P_n\in\{B_n,C_n\}$. The triangulation in $P_n$ distinct from $B_n^+$ (i.e. either $B_n^-$ or $C_n^-$) will be denoted by $T_n$.

Assume that $9\leq{n}\leq10$. It can be seen in Figure \ref{Afigure.10} that edges $\{1,2\}$ and $\{2,3\}$ respectively admit vertex $7$ and vertex $6$ as their links in $T_n$. In addition, triangulation $B_n^+$ contains edge $\{1,3\}$. Hence, invoking Theorem \ref{Acorollary.3} with $U=T_n$, $V=B_n^+$, and $a=1$ provides a vertex $x\in\{1,2\}$ so that $\vartheta(P_n,x)\geq2$. Now denote by $y$ the vertex of $\{1,2\}$ distinct from $x$. It follows that $\{y,3\}$ is a boundary edge of triangulations $T_n{\contract}x$ and $B_n^+{\contract}x$ whose link in $T_n{\contract}x$ is equal to vertex $6$ or to vertex $7$, and whose link in $B_n^+{\contract}x$ is vertex $0$. Therefore, $\vartheta(P_n{\contract}x,y)$ is positive.
\begin{figure}
\begin{centering}
\includegraphics{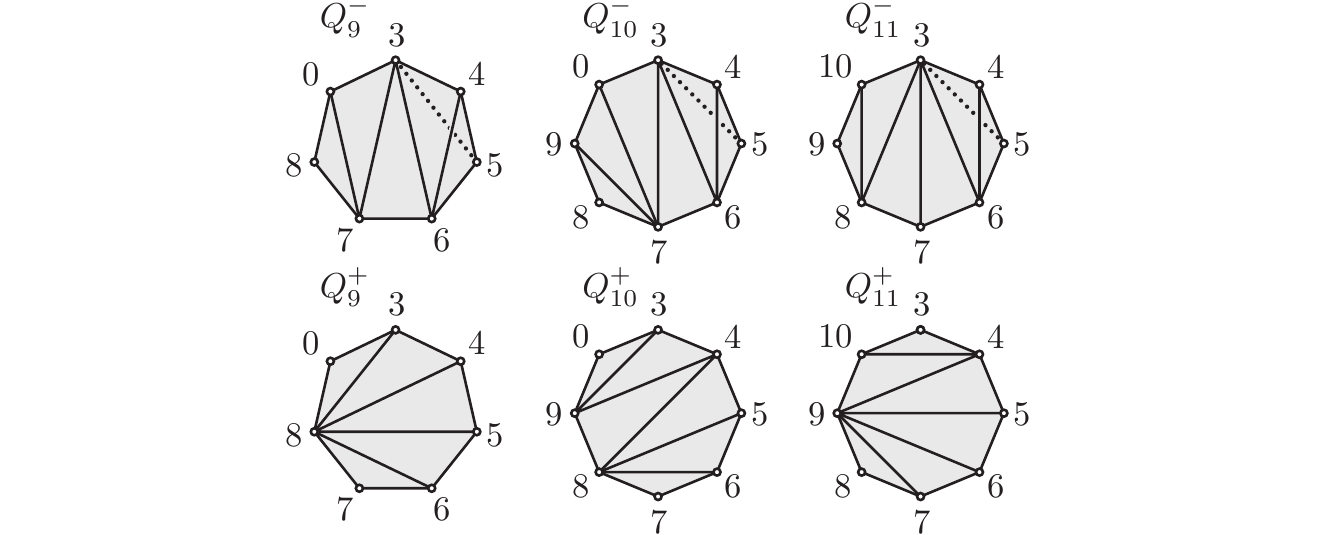}
\caption{Triangulations $Q_n^-$ (top) and $Q_n^+$ (bottom) used the proof of Proposition \ref{Aproposition.9}.}\label{Afigure.11.1}
\end{centering}
\end{figure}

Since $\vartheta(P_n,x)\geq2$ and $\vartheta(P_n{\contract}x,y)\geq1$, invoking Corollary \ref{Acorollary.1} twice yields:
\begin{equation}\label{Aproposition.9.eq.1}
\delta(P_n)\geq\delta(Q_n)+k_n\mbox{,}
\end{equation}
where $Q_n=P_n{\contract}x{\contract}y$ and $k_n=3$. Respectively denote triangulations $T_n{\contract}x{\contract}y$ and $B_n^+{\contract}x{\contract}y$ by $Q_n^-$ and $Q_n^+$. These triangulations are the elements of pair $Q_n$, depicted in the left of Fig. \ref{Afigure.11.1} when $n=9$ and in the center of this figure when $n=10$.

Now assume that $n$ is equal to $11$. One can see in Fig. \ref{Afigure.10} that edges $\{0,1\}$, $\{1,2\}$, and $\{2,3\}$ respectively admit vertices $8$, $7$, and $6$ as their links in $T_{11}$. In addition, edges $\{0,3\}$, $\{1,3\}$, and $\{3,10\}$ are contained in $B_{11}^+$. Hence, invoking Corollary \ref{Acorollary.4} with $a=10$, $U=T_{11}$, and $V=B_{11}^+$ provides the following inequality:
\begin{equation}\label{Aproposition.9.eq.1b}
\delta(P_{11})\geq\delta(Q_{11})+k_{11}\mbox{,}
\end{equation}
where $Q_{11}$ denotes pair $P_{11}{\contract}0{\contract}1{\contract}2$ and $k_{11}$ is equal to $5$. Respectively denote triangulations $T_{11}{\contract}0{\contract}1{\contract}2$ and $B_{11}^+{\contract}0{\contract}1{\contract}2$ by $Q_{11}^-$ and $Q_{11}^+$. These triangulations are the elements of pair $Q_{11}$, shown in the right of Fig. \ref{Afigure.11.1}.

One can see that, whenever $9\leq{n}\leq11$, triangulations $Q_n^-$ and $Q_n^+$ do not have any interior edge in common. As a consequence, their flip distance is not less than the number of their interior edges. In other words, $\delta(Q_9)\geq4$, $\delta(Q_{10})\geq5$ and $\delta(Q_{11})\geq5$. Also observe that if $10\leq{n}\leq11$, no flip can introduce an edge of $Q_n^+$ into $Q_n^-$. Hence, in this case, the flip distance of $Q_n^-$ and $Q_n^+$ is greater than the number of their interior edges, which proves that $Q_{10}$ and $Q_{11}$ have flip distance at least $6$.

Denote $l_9=4$, $l_{10}=6$, and $l_{11}=6$. It has been shown in the last paragraph that:
\begin{equation}\label{Aproposition.9.eq.2}
\delta(Q_n)\geq{l_n}\mbox{.}
\end{equation}

Since $k_n+l_n$ is precisely equal to $2n-11$, the desired result on the flip distance of pair $P_n$ is obtained combining (\ref{Aproposition.9.eq.1}) and (\ref{Aproposition.9.eq.1b}) with (\ref{Aproposition.9.eq.2}). \qed
\end{proof}

One proves that the flip distance of pair $A_n$ is never less than $2n-10$ by invoking the last four propositions and the recursive inequality stated by Theorem \ref{Atheorem.12}:

\begin{theorem}\label{Alemma.11}
For any integer $n$ not less than $3$, $\delta(A_n)\geq2n-10$.
\end{theorem}
\begin{proof}
The proof will proceed by induction on $n$. First observe that the result follows from Proposition \ref{Aproposition.7.9} and from Proposition \ref{Aproposition.8} when $3\leq{n}\leq8$. It will therefore be assumed that $n$ is greater than $8$ in the remainder of the proof. Further assume that for every integer $i\in\{0, ..., n-1\}$, the flip distance of pair $A_i$ is not less than $2i-10$. 

If $9\leq{n}\leq12$, Proposition \ref{Aproposition.7} yields:
\begin{equation}\label{Alemma.11.eq.1}
\delta(A_n)\geq\min(\{\delta(A_{n-1})+2,\delta(A_{n-2})+4,\delta(B_{n-1})+3,\delta(C_{n-1})+3\})\mbox{.}
\end{equation}

By induction, $\delta(A_{n-1})$ is at least $2n-12$ and $\delta(A_{n-2})$ is at least $2n-14$. In addition, since $8\leq{n-1}\leq11$, it follows from Proposition \ref{Aproposition.9} that pairs $B_{n-1}$ and $C_{n-1}$ both have flip distance at least $2n-13$. Hence, the result follows from inequality (\ref{Alemma.11.eq.1}).

Now if $n$ is greater than $12$, then according to Theorem \ref{Atheorem.12},
$$
\delta(A_n)\geq\min(\delta(A_{n-1})+2,\delta(A_{n-2})+4,\delta(A_{n-5})+10,\delta(A_{n-6})+12)\mbox{.}
$$

In this last case, the result is then directly obtained by induction. \qed
\end{proof}

Combining Lemma \ref{Alemma.1} and Theorem \ref{Alemma.11}, one finds that the flip distance of pair $A_n$ is precisely $2n-10$ when $n$ is greater than $12$. According to Definition \ref{Adefinition.3} this also provides the diameter of associahedra of dimension above $9$:
\begin{theorem}\label{Atheorem.main}
The $d$-dimensional associahedron has diameter $2d-4$ when $d>9$.
\end{theorem}

\section{Discussion}
\label{Asection.9}

It has been shown using combinatorial arguments that triangulations $A_n^-$ and $A_n^+$ have flip distance $2n-10$ when $n$ is greater than $12$, thus settling two problems posed in \cite{Sle88}. In particular, the diameter of associahedra is now known in all dimensions. It is first shown in this concluding section that $A_{d+3}^-$ and $A_{d+3}^+$ correspond to maximally distant vertices of the $d$-dimensional associahedron whenever $d$ is distinct from $6$, $7$, and $9$. A problem regarding the possible structure of maximally distant triangulations, originally posed in \cite{Deh10}, is solved next. A discussion on the possible extension of the techniques developed in Section \ref{Asection.2} to other flip-graphs completes the section.

\subsection{The diameter of the $d$-dimensional associahedron when $d\leq9$}

The diameter of the $d$-dimensional associahedron is respectively $0$, $1$, $2$, $4$, $5$, $7$, $9$, $11$, $12$, and $15$ as $d$ ranges from $0$ to $9$ (see \cite{Sle88}). Hence, according to Proposition \ref{Aproposition.7.9} and to Proposition \ref{Aproposition.8}, $A_{d+3}^-$ and $A_{d+3}^+$ not only correspond to maximally distant vertices of the $d$-dimensional associahedron when $d$ is greater than $9$ but also when $d\leq5$. It is worth noting that this property still holds when $d$ is equal to $8$, but fails when $d$ is equal to $6$, $7$, or $9$. Indeed, recall that the flip distance of any two triangulations of a polygon with $n$ vertices is not greater than $2n-10$ when $n$ is greater than $12$. In the case of triangulations $A_n^-$ and $A_n^+$, this observation remains true down to $n=9$, which can be shown using a proof similar to that of Lemma 2 in \cite{Sle88}:

\begin{proposition}\label{Aproposition.10}
For any integer $n$ greater than $8$, $\delta(A_n)\leq2n-10$.
\end{proposition}
\begin{proof}
Let $n$ be an integer greater than $8$. For such values of $n$, vertex $3$ is incident to exactly four interior edges of $A_n^+$. Now observe that if the interior edges of a triangulation $T$ are not all incident to some vertex $x$, it is always possible to introduce a new interior edge incident to $x$ into $T$ by a flip. As $A_n^+$ has $n-3$ interior edges, it can be transformed into the triangulation whose interior edges are all incident to vertex $3$ using $n-7$ flips. Similarly, this triangulation can be reached from $A_n^-$ by a sequence of $n-3$ flips. Hence, triangulations $A_n^-$ and $A_n^+$ have flip distance at most $2n-10$. \qed
\end{proof}

It follows from this proposition that the flip distance of pair $A_{d+3}$ is smaller by one than the diameter of the $d$-dimensional associahedron when $d$ is equal to $6$, $7$, or $9$.

\subsection{A problem on the structure of maximally distant triangulations}

A conjecture from \cite{Deh10} is now disproved. Let $n$ be an integer not less than $3$. Consider the two triangulations obtained by respectively deleting vertices $0$ and $1$ from triangulation $Z_{n+2}$ and vertex $4$ from triangulation $Z_{n+1}$:
$$
D_n^-=Z_{n+2}{\contract}0{\contract}1\mbox{ and }D_n^+=Z_{n+1}{\contract}4\mbox{.}
$$

Further relabel the vertices of these triangulations clockwise from $0$ to $n-1$ with the requirement that vertex $2$ is relabeled $n-1$ in $D_n^-$, and $n-2$ in $D_n^+$. In addition, displace each vertex of $D_n^+$ to the vertex of $D_n^-$ with the same label in order to obtain triangulations of the same polygon. Triangulations $D_n^-$ and $D_n^+$ are shown in Fig. \ref{Afigure.11} depending on the parity of $n$, when $n$ is greater than $5$.
\begin{figure}
\begin{centering}
\includegraphics{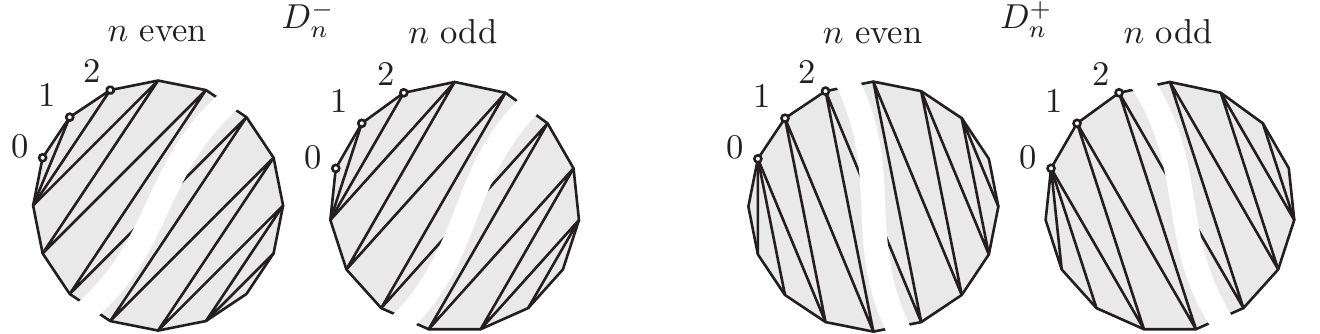}
\caption{Triangulations $D_n^-$ (left) and $D_n^+$ (right), depicted when $n$ is greater than $5$.}\label{Afigure.11}
\end{centering}
\end{figure}
Each of these triangulations admits a unique comb with three teeth. Call:
$$
D_n=\{D_n^-,D_n^+\}\mbox{.}
$$

It is conjectured in \cite{Deh10} that $D_n$ has flip distance $2n-10$ when $n$ is greater than $9$. Note that this conjecture is originally stated in the formalism of binary trees. While this particular statement turns out to be false, the insight provided in \cite{Deh10} on the possible structure of maximally distant vertices of the associahedra is partly correct. Indeed, pairs $A_n$ and $D_n$ are similar, except that the two triangulations in $A_n$ are not symmetric and have combs at both ends of the zigzag. The two triangulations in pair $D_n$ are not always maximally distant, though:

\begin{proposition}\label{Aproposition.11}
For any integer $n$ greater than $19$, $\delta(D_n)\leq\delta(D_{n-16})+31$.
\end{proposition}
\begin{proof}
Let $n$ be an integer greater than $19$. Because of this condition on $n$, the following fourteen edges are interior edges of triangulation $D_n^-$:
$$
\begin{array}{lllll}
\{2,n-1\}\mbox{,}&\{1,n-1\}\mbox{,} & \{4,n-2\}\mbox{,} & \{4,n-3\}\mbox{,} & \{5,n-3\}\mbox{,}\\
\{5,n-4\}\mbox{,} & \{6,n-4\}\mbox{,} & \{7,n-5\}\mbox{,} & \{6,n-5\}\mbox{,} & \{7,n-6\}\mbox{,}\\
\{8,n-6\}\mbox{,} & \{10,n-8\}\mbox{,} & \{9,n-8\}\mbox{,} & \{9,n-7\}\mbox{.} & \\
\end{array}
$$

Let $(U_i)_{0\leq{i}\leq14}$ be the path that starts at triangulation $D_n^-$ and that successively flips these edges in the order in which they are listed above, from left to right and then from top to bottom. Similarly, since $n$ is greater than $19$, the following seventeen edges are interior edges of triangulation $D_n^+$:
$$
\begin{array}{llllll}
\{2,n-6\}\mbox{,} & \{2,n-5\}\mbox{,} & \{1,n-5\}\mbox{,} & \{1,n-4\}\mbox{,} & \{0,n-4\}\mbox{,} & \{0,n-3\}\mbox{,}\\
\{0,n-2\}\mbox{,} & \{4,n-8\}\mbox{,} & \{4,n-7\}\mbox{,} & \{5,n-8\}\mbox{,} & \{5,n-9\}\mbox{,} & \{5,n-7\}\mbox{,}\\
\{3,n-7\}\mbox{,} & \{7,n-11\}\mbox{,} & \{7,n-10\}\mbox{,} & \{6,n-10\}\mbox{,} & \{6,n-9\}\mbox{.} &\\
\end{array}
$$

Call $(V_i)_{0\leq{i}\leq17}$ the path that starts with $D_n^+$ and that flips these edges in the order in which they are listed above, from left to right and then from top to bottom.

In order to visualize paths $(U_i)_{0\leq{i}\leq14}$ and $(V_i)_{0\leq{i}\leq17}$, triangulations $U_7$, $U_{11}$ and $U_{14}$ are depicted in the top of Fig. \ref{Afigure.12} and triangulations $V_7$, $V_{13}$, and $V_{17}$ are shown in the bottom of the same figure.
\begin{figure}
\begin{centering}
\includegraphics{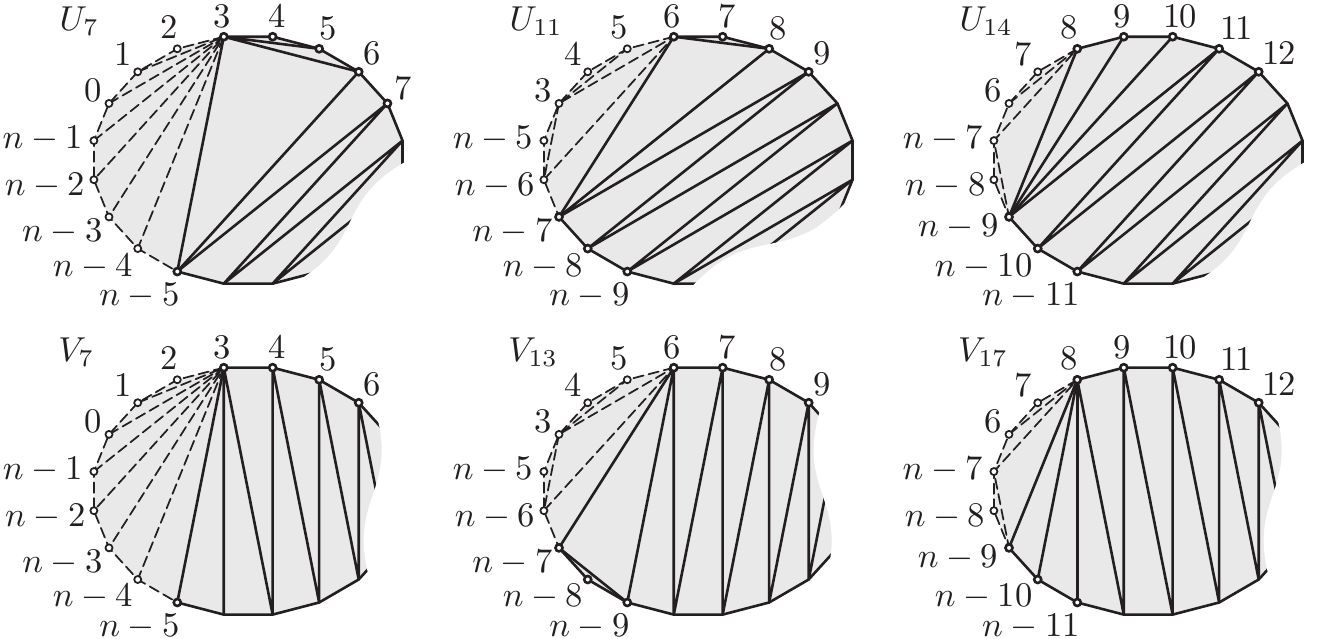}
\caption{Triangulations $U_7$, $U_{11}$, $U_{14}$, $V_7$, $V_{13}$, and $V_{17}$ defined in the proof of Proposition \ref{Aproposition.11}.}\label{Afigure.12}
\end{centering}
\end{figure}
One can see that the edges introduced by the first seven flips along each of these two paths are incident to vertex $3$. In addition, $U_7$ and $V_7$ have seven interior edges in common. The edges sketched as dashed lines in Fig. \ref{Afigure.12} are common to all the triangulations subsequently visited by both paths, and they are omitted in this figure from the sketches of $U_{11}$, $U_{14}$, $V_{13}$, and $V_{17}$.

As can be seen in the center of the figure, five more interior edges are common to $U_{11}$ and $V_{13}$. Again, the dashed edges are common to all the triangulations later visited by both paths, and are omitted from the sketches of $U_{14}$ and $V_{17}$.

Finally, it can be seen in the right of Fig. \ref{Afigure.12} that removing from triangulations $U_{14}$ and $V_{17}$ all the edges that have a vertex less than $8$ or greater than $n-9$ results in a pair of triangulations isomorphic to $D_{n-16}$. As all the removed edges are common to $U_{14}$ and $V_{17}$, this proves that $\delta(D_n)$ is not greater than $\delta(D_{n-16})+31$. \qed
\end{proof}

Therefore, not only is $\delta(D_n)$ smaller than $\delta(A_n)$ when $n$ is large enough, but the difference between these two quantities increases at least linearly with $n$:

\begin{theorem}
The flip distance of pair $D_n$ is at most $\displaystyle\frac{31}{16}n+O(1)$.
\end{theorem}

\subsection{On the diameter of other flip-graphs}

Along with the many constructions of associahedra \cite{Bil90,Ceb13,Cha02,Gel90,Hol07,Lod04,Pos09} that followed the ones by Mark Haiman and Carl Lee, several generalizations of these polytopes have been discovered \cite{Bil90,Bil92,Dev09,Fom03,Gel90}. Among them, one finds the secondary polytopes \cite{Bil90,Gel90} that share many properties of associahedra. In particular, their vertices are in one-to-one correspondence with the regular triangulations of given finite sets of points and their edges correspond to the flips between these triangulations. Thus, it is natural to ask whether the techniques developed in this article could be generalized to such triangulations, in the hope for results on the diameter of secondary polytopes.

Recall that deleting a vertex in some triangulation of a polygon always results in a triangulation. This is not true any more in the case of more general triangulations, even already in dimension $2$ for triangulations with interior vertices. A preliminary search for point configurations with good properties regarding the deletion of their boundary faces could be a first step in the investigation of this more general problem. Note that, along with these difficulties, the number of vertices of secondary polytopes is unknown in general, and no reasonable upper bound is available on their diameter. The diameter of related polytopes, that are not secondary polytopes, may be found using ideas similar to the ones developed here. An example of such polytopes are cyclohedra whose $1$-skeleton is isomorphic to the flip-graph of centrally symmetric triangulations. In this case, deletions would have to remove two opposite vertices.

Finally, connected but not necessarily polytopal flip-graphs could also be investigated using results similar to the ones given in Section \ref{Asection.2}. In particular, the maximal flip distances between multi-triangulations \cite{Pil09} may be explored. Note that possibly sharp upper bounds on these distances are already known \cite{Pil09}. Another interesting question is that of the maximal flip distance between topological triangulations of an orientable surface with an arbitrary number of boundaries and arbitrary genus (considered up to homeomorphism in order to keep this distance finite).

\bibliographystyle{model1b-num-names}
\bibliography{Associahedra}







\end{document}